\documentclass{article}
\usepackage{amsmath}
\usepackage{amssymb}
\usepackage{amsthm}
\usepackage{enumitem}
 \usepackage{young}

\newcommand{\B}{B_n}
\newcommand{\vB}{v \negthinspace B_n}

\newcommand{\fB}{f \negthinspace B_n}

\newcommand{\wB}{w \negthinspace B_n}
\newcommand{\PB}{P \negthinspace B_n}
\newcommand{\pb}{\mathfrak{pb}_n}
\newcommand{\pbq}{\mathfrak{pb}_n^!}
\newcommand{\PwB}{P \negthinspace w \negthinspace B_n}
\newcommand{\PvB}{P \negthinspace v \negthinspace B_n}
\newcommand{\PvBii}{P \negthinspace v \negthinspace B_2}
\newcommand{\PvBiii}{P \negthinspace v \negthinspace B_3}
\newcommand{\pv}{\mathfrak{pvb}_n}

\newcommand{\pvsz}{\mathfrak{pvb}_{n,\sigma}(z)}
\newcommand{\pvs}{\mathfrak{pvb}_{n,\sigma}}

\newcommand{\pvq}{\mathfrak{pvb}_n^!}
\newcommand{\pvqk}{\mathfrak{pvb}_n^{!k}}

\newcommand{\pvqni}{\mathfrak{pvb}_n^{!n-1}}

\newcommand{\pvqsz}{\mathfrak{pvb}_{n,\sigma}^!(z)}
\newcommand{\pvqsk}{\mathfrak{pvb}_{n,\sigma}^{!k}}

\newcommand{\pvqz}{\mathfrak{pvb}_n^!(z)}
\newcommand{\PfB}{P \negthinspace f \negthinspace B_n}
\newcommand{\PfBii}{P \negthinspace f \negthinspace B_2}

\newcommand{\pf}{\mathfrak{pfb}_n}

\newcommand{\pfq}{\mathfrak{pfb}_n^!}
\newcommand{\pfqk}{\mathfrak{pfb}_n^{!k}}
\newcommand{\pfqi}{\mathfrak{pfb}_n^{!i}}

\newcommand{\pfqsz}{\mathfrak{pfb}_{n,\sigma}^!(z)}
\newcommand{\pfsz}{\mathfrak{pfb}_{n,\sigma}(z)}

\newcommand{\paq}{\mathfrak{pab}_n^{!}}

\newcommand{\Hspv}{H^*(\PvB,\Q)}
\newcommand{\Hipv}{H^i(\PvB,\Q)}
\newcommand{\Hkpv}{H^k(\PvB,\Q)}
\newcommand{\Hspf}{H^*(\PfB,\Q)}
\newcommand{\Hipf}{H^i(\PfB,\Q)}
\newcommand{\Hkpf}{H^k(\PfB,\Q)}

\newcommand{\Hipb}{H^i(\PB,\Q)}
\newcommand{\Hkpb}{H^k(\PB,\Q)}
\newcommand{\Hsv}{H^*(\vB,\Q)}
\newcommand{\Hiv}{H^i(\vB,\Q)}

\newcommand{\Hkv}{H^k(\vB,\Q)}

\newcommand{\Hif}{H^i(\fB,\Q)}

\newcommand{\Hkf}{H^k(\fB,\Q)}

\newcommand{\Hkpw}{H^k(\PwB,\Q)}
\newcommand{\Hspw}{H^*(\PwB,\Q)}

\newcommand{\ua}{\underline{a}}
\newcommand{\Ka}{K_{\ua}}
\newcommand{\ch}{\text{ch}}
\newcommand{\ois}{\omega_i^*}
\newcommand{\oi}{\omega_i}

\newcommand{\Rq}{R^\perp}

\newcommand{\Sn}{S_n}
\newcommand{\Q}{\mathbb{Q}}

\newcommand{\Z}{\mathbb{Z}}
\newcommand{\R}{\mathbb{R}}
\newcommand{\N}{\mathbb{N}}
\newcommand{\QG}{\Q G}
\newcommand{\grQG}{gr \Q G}
\newcommand{\lra}{\longrightarrow}
\newcommand{\ra}{\rightarrow}
\newcommand{\hra}{\hookrightarrow}
\newcommand{\thra}{\twoheadrightarrow}

\newcommand{\vij}{v_{ij}}
\newcommand{\vik}{v_{ik}}
\newcommand{\vjk}{v_{jk}}

\newcommand{\uij}{u_{ij}}
\newcommand{\uik}{u_{ik}}
\newcommand{\ujk}{u_{jk}}

\newcommand{\ai}{\alpha_i}

\newcommand{\rij}{r_{ij}}
\newcommand{\rik}{r_{ik}}
\newcommand{\rjk}{r_{jk}}
\newcommand{\rji}{r_{ji}}

\newcommand{\rkj}{r_{kj}}
\newcommand{\rkl}{r_{kl}}

\newcommand{\w}{\wedge}
\newcommand{\dt}{d_{\tau}}
\newcommand{\ttau}{t_{\tau}}
\newcommand{\alt}{\alpha_{\tau}}
\newcommand{\chis}{\chi_{\sigma}}
\newcommand{\chism}{\chi_{\sigma}(m)}
\newcommand{\Gi}{\Gamma_1}

\newcommand{\Gii}{\Gamma_i}
\newcommand{\Gr}{\Gamma_r}
\newcommand{\gi}{\gamma_1}

\newcommand{\gj}{\gamma_j}
\newcommand{\gki}{\gamma_{k_i}}
\newcommand{\mB}{\mathcal{B}}
\newcommand{\As}{A_{\sigma}}
\newcommand{\Aqs}{A_{\sigma}^!}
\newcommand{\Aq}{A^!}
\newcommand{\Xni}{X^n_i}
\newcommand{\Xnii}{X^n_{i+1}}
\newcommand{\Xnio}{X^n_1}
\newcommand{\Xnit}{X^n_2}
\newcommand{\Xnin}{X^n_{n-1}}
\newcommand{\si}{\sigma_i}
\newcommand{\sii}{\sigma_{i+1}}
\newcommand{\sj}{\sigma_j}
\newcommand{\Aij}{A_{ij}}
\newcommand{\aij}{a_{ij}}
\newcommand{\aji}{a_{ji}}
\newcommand{\aik}{a_{ik}}
\newcommand{\ajk}{a_{jk}}
\newcommand{\akl}{a_{kl}}
\newcommand{\ind}{\mathrm{Ind}}

\newcommand{\mS}{\mathcal{S}}
\newcommand{\bmS}{\overline{\mathcal{S}}}
\newcommand{\mSa}{\mathcal{S}_{\alpha}}
\newcommand{\pvqs}{\mathfrak{pvb}_n^{!\mS}}
\newcommand{\pvqsp}{\mathfrak{pvb}_n^{!\mS'}}
\newcommand{\pvqsa}{\mathfrak{pvb}_n^{!\mS_{\alpha}}}
\newcommand{\pvqsat}{\mathfrak{pvb}_n^{!\mS_{\alpha}\otimes}}
\newcommand{\pfqsa}{\mathfrak{pfb}_n^{!\mS_{\alpha}}}
\newcommand{\pfqsat}{\mathfrak{pfb}_n^{!\mS_{\alpha}\otimes}}
\newcommand{\pfqs}{\mathfrak{pfb}_n^{!\mS}}
\newcommand{\stab}{\mathrm{Stab}}
\newcommand{\Sai}{S_{\ai}}
\newcommand{\cln}{c_{\lambda,n}}

\newtheorem{example}{Example}

\newtheorem{proposition}{Proposition}

\newtheorem{theorem}{Theorem}

\newtheorem{lemma}{Lemma}
\newtheorem{corollary}{Corollary}
\newtheorem{definition}{Definition}

\usepackage[all,knot,poly]{xy}

\begin{document}

\title{On the Action of the Symmetric Group on the Cohomology of Groups Related to (Virtual) Braids}
\author{Peter Lee}
\maketitle

\begin{abstract}
In this paper we consider the cohomology of four groups related to the virtual braids of \cite{Kau} and \cite{GPV}, namely the pure and non-pure virtual braid groups ($\PvB$ and $\vB$, respectively), and the pure and non-pure flat braid groups ($\PfB$ and $\fB$, respectively).  The cohomologies of $\PvB$ and $\PfB$ admit an action of the symmetric group $S_n$.  We give a description of the cohomology modules $\Hipv$ and $\Hipf$ as sums of $S_n$-modules induced from certain one-dimensional representations of specific subgroups of $S_n$.  This in particular allows us to conclude that $\Hipv$ and $\Hipf$ are uniformly representation stable, in the sense of \cite{C-F}.  We also give plethystic formulas for the Frobenius characteristics of these $S_n$-modules.  We then derive a number of constraints on which $S_n$ irreducibles may appear in $\Hipv$ and $\Hipf$.  In particular, we show that the multiplicity of the alternating representation in $\Hipv$ and $\Hipf$ is identical, and moreover is nil for sufficiently large $n$.  We use this to recover the (previously known) fact that the multiplicity of the alternating representation in $\Hipb$ is nil (here $\PB$ is the ordinary pure braid group).  We also give an explicit formula for $\Hiv$ and show that $\Hif=0$.  Finally, we give Hilbert series for the character of the action of $S_n$ on $\Hipv$ and $\Hipf$.  An extension of the standard `Koszul formula' for the graded dimension of Koszul algebras to graded characters of Koszul algebras then gives Hilbert series for the graded characters of the respective quadratic dual algebras.
\end{abstract}

\tableofcontents

\section{Introduction}
In this paper we will be concerned with the action of the symmetric groups $S_n$ on the cohomology of the groups $\PvB$ and $\vB$ (pure and non-pure virtual braid groups) and $\PfB$ and $\fB$ (pure and non-pure flat braid groups), and on certain related algebras.  The context and nature of our investigation is the following.

Let $\{G_n\}_{n\in \N}$ be a family of groups equipped with maps $\phi_n: G_n \ra G_{n+1}$, and let $\{P_n\}_{n\in\N}$ be a family of normal subgroups, $P_n \trianglelefteq G_n$, such that the quotient groups $Q_n:=G_n/P_n$ are finite and there are diagrams

\[
\xy
(0,0)*+{0} = "1";
(15,0)*+{P_n} = "2";
(30,0)*+{G_n} = "3";
(45,0)*+{Q_n} = "4";
(60,0)*+{0} = "5";
{\ar@{->} "1"; "2"};
{\ar@{->} "2"; "3"};
{\ar@{->} "3"; "4"};
{\ar@{->} "4"; "5"};
(0,-15)*+{0} = "11";
(15,-15)*+{P_{n+1}} = "12";
(30,-15)*+{G_{n+1}} = "13";
(45,-15)*+{Q_{n+1}} = "14";
(60,-15)*+{0} = "15";
{\ar@{->} "11"; "12"};
{\ar@{->} "12"; "13"};
{\ar@{->} "13"; "14"};
{\ar@{->} "14"; "15"};
{\ar@{->}^{\phi_n^{res}} "2"; "12"};
{\ar@{->}^{\phi_n} "3"; "13"};
{\ar@{->}^{\phi_n^S} "4"; "14"};
\endxy
\]
which commute and have exact rows.  In this paper, the $Q_n$ will always be the symmetric groups $S_n$.

There is a circle of questions naturally associated to such a family.  Without attempting to be exhaustive, these include (fixing some coefficient ring $R$, and taking $k\geq 1$):

\begin{enumerate}

\item \emph{(Co-)homological Stability and Identification:}  Can the $H^k(G_n,R)$ be identified exactly? Do these modules stabilize for sufficiently large $n$?  (there are analogous questions pertaining to homology.)

\item \emph{Decomposition:}  There is a natural action of $Q_n$ on $P_n$, and an induced action on the cohomology rings $H^k(P_n,R)$.  What is the decomposition of the $H^k(P_n,R)$ as irreducible $Q_n$-modules?  In the alternative, or in addition, is there a succinct expression for the $H^k(P_n,R)$ as (possibly, sums of) induced modules?

\item \emph{Representation Stability:}  It is common for the groups $P_n$ not to be cohomologically stable.  However, following \cite{C-F}, we can ask whether the decomposition of the $H^k(P_n,R)$ as irreducible $Q_n$-modules stabilizes for sufficiently large $n$ (the exact definition of representation stability is given in Subsection \ref{RepStabSubsection}).

\item \emph{Graded Characters:}  Give Hilbert series describing the graded characters for the action of $S_n$ on $H^*(P_n,R)$.

\item \emph{Associated Graded and Malcev Lie Algebras:}  If the group algebras $RP_n$ are filtered by powers of their augmentation ideals, denote by $grRP_n$ the associated graded algebras.  By \cite{Quillen}, if (for instance) we take $R=\Q$, then $grRP_n$ is in fact the universal enveloping algebra of the Malcev Lie algebra of $P_n$.  Attempt to answer questions 1 - 4 for $grRP_n$ (whenever these questions make sense).

\end{enumerate}

There are several families of groups for which homological stability has been established, including Artin's braid groups $\{\B\}$, certain mapping class groups, and outer automorphism groups of free groups;  many of the resulting stable homology groups have actually been determined (see for instance \cite{C}, \cite{V}, \cite{C-F} and the references therein).\footnote{Much of the historical background in this introduction was drawn from \cite{C-F}, where a much more extensive discussion can be found.}  The rational cohomology of the string motion groups were shown to be trivial in \cite{W}, in degrees $k\geq 1$.

Specializing to the case when $Q_n$ is a finite Coxeter group and $G_n$ is the corresponding generalized braid group, it appears that little is known in general about the decomposition of $H^k(P_n,R)$ as irreducible $Q_n$-representations (although partial results relating to ordinary pure braid groups appear in \cite{L-S}).  However, an expression for $H^k(P_n,R)$ as a direct sum of representations induced from one-dimensional representations of certain centralizers in $S_n$ was given in the case $P_n=\PB$ in \cite{L-S}.

It was already clear from the computation of $\Hkpb$ in \cite{Arnold} that the dimension of these algebras grows without bound as $n$ increases, so that the $\PB$ are not cohomologically stable.  However, in \cite{C-F} T. Church and B. Farb introduced their notion of representation stability and showed that the cohomology of the groups $\PB$ is representation stable.  Representation stability for the cohomology of the pure string motion groups was established in \cite{W}.

Lehrer \cite{L}, and Blair and Lehrer \cite{B-L}, gave Hilbert series describing the graded character of the action of Coxeter groups on the cohomology of the complements of the related hyperplanes;  when the Coxeter group is $S_n$, this is equivalent to the graded character of $S_n$ on the cohomology of $\PB$.

There seems to have been little attention paid to the action of $Q_n$ on the associated graded algebras $grRP_n$, although as we will see their graded characters can be obtained from those of $H^*(P_n,\Q)$ when these algebras are Koszul.

Having outlined the types of question in which we shall be interested, we now turn to the groups $\PvB$, $\vB$, $\PfB$ and $\fB$, for which we will seek to answer those questions.  $\PvB$ is defined as the group generated by symbols $R_{ij}$, $1 \leq i \neq j \leq n$, subject to the Yang-Baxter (or Reidemeister III) relations and certain commutativities:

\begin{align}
R_{ij} R_{ik} R_{jk} & = R_{jk} R_{ik} R_{ij} \label{relations1} \\
R_{ij} R_{kl} & = R_{kl} R_{ij}, \quad \quad
\label{relations2}
\end{align}
with $i,j,k,l$ distinct.  The groups $\PfB$ are given by the same presentation but subject to the additional relations $R_{ij}R_{ji}=1$.\footnote{The definition of $\vB$ and $\fB$ is given in Subsection \ref{GroupDefsSubs}.}

The interest in the $S_n$-action on the cohomology of $\PvB$ and $\PfB$ derives partly from these groups' close relation to $\PB$.  In particular, as pointed out by Bartholdi, Enriques, Etingof and Rains in \cite{BEER}\footnote{In \cite{BEER}, $\PvB$ is referred to as the quasi-triangular group $QTr_n$, while $\PfB$ is referred to as the triangular group $Tr_n$.} there is an `almost' exact sequence of groups:

\begin{equation*}
1 \ra \PB \overset{\Psi_n}{\lra} \PvB \overset{\Pi_n}{\lra} \PfB \lra 1
\end{equation*}

`Almost' exact means that this sequence is exact on the left and right, and moreover the kernel of $\Pi_n$ is the normal closure of the image of $\Psi_n$.

In a similar vein, there is a complex at the level of cohomology:

\begin{equation}
\label{CohomExactSequence}
0 \ra \Hkpf \ra \Hkpv \ra \Hkpb \ra 0
\end{equation}
which is also exact on the left and right (though usually not in the middle).  Thus in suitable circumstances one can use information about the $S_n$ actions on two of these cohomologies to deduce information about the $S_n$ action on the third (see Theorem \ref{thm:PBAltMultThm} below).

An additional reason for the interest in $\PvB$, particularly, is its relation with the pure (untwisted) string motion group $\PwB$ (also known as the pure welded braid group, or the group of pure symmetric automorphisms of a free group of rank n). $\PwB$ is the quotient of $\PvB$ by the further relations $R_{ij}R_{ik}=R_{ik}R_{ij}$, when $i,j,k$ are distinct.

Much as for $\Hkpb$ and $\Hkpw$, there is an action of the symmetric group $\Sn$ on the cohomology modules $\Hkpv$ and $\Hkpf$.  A key result of this paper is the following theorem (reminders concerning wreath products, their representations, and related notation appears in Subsection \ref{WreathProdSubsection}):

\newtheorem*{thm:pvqkDecompositionThm}{Theorem \ref{thm:pvqkDecompositionThm}}
\begin{thm:pvqkDecompositionThm}
As an $S_n$-module, $\Hkpv$ has the following decomposition:

\begin{equation*}
\Hkpv = \bigoplus_{\substack{\alpha \vdash n \\ l(\alpha)=n-k}} \ind_{\prod_i 1\wr \Sai}^{S_n} \bigotimes_i \ \ois (-1)^{i-1}
\end{equation*}
where the direct sum is over partitions $\alpha= \sum_{i=1}^{n} i \alpha_i$ of $n$, with $\alpha_i$ non-negative, with $l(\alpha)=n-k$ parts.  Also:

\begin{enumerate}
\item In the product $\prod_i 1\wr \Sai$, $1\wr \Sai$ means $S_1\wr \Sai$, understood as the subgroup $1 \wr \Sai \leq S_i \wr S_{\alpha_i} \leq S_{i\ai}$, and $S_{i\ai}$ is viewed as the subgroup of $S_n$ which permutes the indices $\{1+\sum_{j=1}^{i-1} j \alpha_j, \cdots, i\ai +\sum_{j=1}^{i-1} j \alpha_j\}$;
\item $(-1)^{i-1}$ denotes the 1-dimensional representation of $\Sai$ which is either trivial or alternating, according to the parity of $i-1$;
\item For each $i=1,\dots,n$, $\ois (-1)^{i-1}$ is the pullback of $(-1)^{i-1}$ along the projection $\oi: 1\wr \Sai \thra \Sai$; and
\item $\bigotimes_i$ denotes the outer product of representations.
\end{enumerate}
\end{thm:pvqkDecompositionThm}

In a similar vein we have the following decomposition of $\Hkpf$.  We retain the notation introduced in the previous theorem, except that $\oi$ is now the projection $\oi: S_i\wr \Sai \thra \Sai$, and $\text{Alt}_i$ is the alternating representation of $S_i$.  We then write $\text{Alt}_i\wr \ois (-1)^{i-1}$ for the wreath product of representations (see Subsection \ref{WreathProdSubsection} for the definition). Then:

\newtheorem*{thm:pfqkDecompositionThm}{Theorem \ref{thm:pfqkDecompositionThm}}
\begin{thm:pfqkDecompositionThm}
As an $S_n$-module, $\Hkpf$ has the following decomposition:
\begin{equation*}
\Hkpf = \bigoplus_{\substack{\alpha \vdash n \\ l(\alpha)=n-k}} \ind_{\prod_i S_i\wr S_{\alpha_i}}^{S_n} \bigotimes_i \ \mathrm{Alt}_i\wr \ois (-1)^{i-1}
\end{equation*}
where the direct sum is over partitions $\alpha= \sum_{i=1}^{n} i \alpha_i$ of $n$, with $\alpha_i$ non-negative, and $l(\alpha)=n-k$ parts.
\end{thm:pfqkDecompositionThm}

We give an alternative formulation of these decomposition theorems in terms of plethysm of Frobenius characteristics in Theorems \ref{thm:pvkPlethysm} and \ref{thm:Plethysm} below.

As corollaries, we derive the fact that both $\Hkpv$ and $\Hkpf$ are uniformly representation stable, with stability range $n\geq 4k$, for all $k\geq 1$ (see Corollary \ref{pvqRepStabCor} and Corollary \ref{pfqRepStabCor}).

\newpage

Following the `transfer' argument applied by \cite{C-F} in the context of the (ordinary) braid group, we compute the cohomology groups $\Hkv$ and $\Hkf$, for sufficiently large $n$.  More precisely, we obtain:\footnote{For a module $V$ over a group $G$, we write $V^G$ for the $G$-invariant submodule of $V$.}

\newtheorem*{thm:vBCohomolThm}{Theorem \ref{thm:vBCohomolThm}}
\begin{thm:vBCohomolThm}
For $k,n \in \N$ and $n>k\geq 1$,
\begin{equation*}
\Hkpv^{S_n} \cong H^k(\vB,\Q) \cong \Q^{P(k)}
\end{equation*}
where $P(k)$ is the number of partitions of $k$ into at most $n-k$ parts, in which each odd part (if any) appears just once.

If in fact $n\geq 2k$, then $P(k)$ is the total number of partitions of $k$, in which each odd part (if any) appears just once.
\end{thm:vBCohomolThm}

We recall\footnote{See A006950, \emph{The On-Line Encyclopedia of Integer Sequences}, http://oeis.org/?language=english.} that a generating function for the number of partitions of $i$ with no repeated odd parts is:

\begin{equation*}
P(z)=\prod_{k>0} \frac{(1+z^{2k-1})}{(1-z^{2k})}
\end{equation*}

For $\Hkf$ we obtain:

\newtheorem*{thm:fBCohomThm}{Theorem \ref{thm:fBCohomThm}}
\begin{thm:fBCohomThm}
For $k,n \in \N$ and $n>k\geq 1$,
\begin{equation*}
\Hkpf^{S_n} \cong \Hkf \cong 0
\end{equation*}
\end{thm:fBCohomThm}

We can also consider the multiplicity of the alternating representation in $\Hipv$ and $\Hipf$.  We obtain the following theorem:\footnote{For an $S_n$-module $V$, we denote by $V^{Alt}$ the projection of $V$ onto the (sum of the) alternating submodules in $V$.}

\newtheorem*{thm:AltMultThm}{Theorem \ref{thm:AltMultThm}}
\begin{thm:AltMultThm}
For all $n>k\geq 1$,
\begin{equation*}
\Hkpv^{Alt} \cong \Hkpf^{Alt}
\end{equation*}
and both vanish for sufficiently large $n$ ($n\geq 2(k+1)$ will do.)
\end{thm:AltMultThm}

Then, using the complex (\ref{CohomExactSequence}) (and the fact that all the maps in the complex commute with the $S_n$ action), we recover the following theorem (originally proved in \cite{H} by different means) as an immediate corollary:

\newtheorem*{thm:PBAltMultThm}{Theorem \ref{thm:PBAltMultThm}}
\begin{thm:PBAltMultThm}
For all $n>k\geq 1$,
\begin{equation*}
\Hkpb^{Alt} =0
\end{equation*}
\end{thm:PBAltMultThm}

We can reformulate the decomposition for $\Hkpv$ and $\Hkpf$ in Theorems \ref{thm:pvqkDecompositionThm} and \ref{thm:pfqkDecompositionThm} in terms of plethysm, as follows (we denote ch$V$ the Frobenius characteristic of an $S_n$-module $V$):

\newtheorem*{thm:pvkPlethysm}{Theorem \ref{thm:pvkPlethysm}}
\begin{thm:pvkPlethysm}
\begin{equation*}
\mathrm{ch} \Hkpv = \bigoplus_{\substack{\ua \vdash k \\ l(\ua) \leq n-k}} \prod_t v_{t,a_t} [\mathrm{ch Reg}_{t+1}] \cdot h_{n-k-l(\ua)}
\end{equation*}
where the sum is over partitions $\ua: k=\sum_{t=1}^k ta_t,\ a_t \in \N$, of $k$, with length $l(\ua)$ no greater than $n-k$; and:

\begin{itemize}
\item $h_p=s_{(p)}$ is the Schur function for the partition $p=p$ of $p\in \N$;
\item $v_{t,a_t}:=h_{a_t}$ if $t$ is even, and $v_{t,a_t}:=e_{a_t}$ if $t$ is odd, where $e_p=s_{(1,\dots, 1)}$ is the Schur function for the partition of $p\in \N$ consisting of $p$ ones;
\item $v_{t,a_t}[\mathrm{ch Reg}_{t+1}]$ denotes plethysm of Schur functions.
\end{itemize}
\end{thm:pvkPlethysm}

\newtheorem*{thm:Plethysm}{Theorem \ref{thm:Plethysm}}
\begin{thm:Plethysm}
\begin{equation*}
\mathrm{ch} \Hkpf = \bigoplus_{\substack{\ua \vdash k \\ l(\ua) \leq n-k}} \prod_t v_{t,a_t} [e_{t+1}] \cdot h_{n-k-l(\ua)}
\end{equation*}
where the sum is over partitions $\ua: k=\sum_{t=1}^k ta_t,\ a_t \in \N$, of $k$, with length $l(\ua)$ no greater than $n-k$ (and we otherwise use the notation given in the previous theorem)
\end{thm:Plethysm}

Using this reformulation of Theorem \ref{thm:pfqkDecompositionThm}, we readily derive the following restrictions on which irreducibles may appear in $\Hkpf$:

\newtheorem*{thm:LamExceedsK}{Theorem \ref{thm:LamExceedsK}}
\begin{thm:LamExceedsK}
Let $\lambda: (\lambda_0 \geq \lambda_1 \geq \dots \geq \lambda_r >0)$ be a partition of $n$, that is:  $\sum_{i=0}^r \lambda_i= n$.  Denote by $V(\lambda)$ the $S_n$-irreducible corresponding to this partition.  Then $V(\lambda)$ occurs in $\Hkpf$ only if $|\lambda| := n-\lambda_0 \geq k$.\footnote{The notation $V(\lambda)$ for an irreducible representation of $S_n$ is further explained in subsection \ref{RepStabSubsection}.}
\end{thm:LamExceedsK}

\newtheorem*{thm:NoVp}{Theorem \ref{thm:NoVp}}
\begin{thm:NoVp}
Let $(p_0\geq p >0)$ be a partition of $n$ with two non-zero parts (that is, $n=p_0+p$).  Denote by $V(p)$ the $S_n$-irreducible corresponding to this partition.  As a representation of $S_n$, $\Hkpf$ does not include the irreducible $V(p)$, for any $p\geq 1$, except for $1$ copy of $V(1)$ when $k=1$.
\end{thm:NoVp}

Finally, we give Hilbert series encoding the characters of the $S_n$-actions on $\Hspv$, $\Hspf$, and their quadratic dual algebras.  For reasons to be explained below, we will denote by $\pvqsk$ the character of the $S_n$-action on the graded components $\Hkpv$, evaluated at any $\sigma \in S_n$.  We define:

\begin{equation*}
\pvqsz := \sum_{k\geq 0} \pvqsk z^k
\end{equation*}
We take $\mathfrak{pvb}_1^{!}=\Q$ to be the trivial representation, so that $\mathfrak{pvb}_1^{!}(z)=1$.

\newtheorem*{thm:PvBTheorem}{Theorem \ref{thm:PvBTheorem}}
\begin{thm:PvBTheorem}

1). Let $\sigma\in S_n$ have cycle type corresponding to a `homogeneous' partition of $n$, that is $n=k+ \dots + k$ (with $\alpha_k$ summands), for some $k, \alpha_k \geq 1$.  Then:

\begin{equation*}
\pvqsz =  \sum_{0\leq \beta \leq \alpha_k} L(\alpha_k,\beta) (-1)^{(\alpha_k-\beta)(k-1)} k^{(\alpha_k-\beta)} z^{(\alpha_k-\beta)k}
\end{equation*}
where the $L(p,q)$ stand for Lah numbers, which count the number of unordered partitions of $[p]:=\{1, \dots, p\}$ into $q$ ordered subsets.  This extends the case $\sigma = 1 \in S_n$, which gives the Hilbert series for the graded dimensions of $\pvq$, which were derived in \cite{BEER}.

2).  Now let $\sigma\in S_n$ have cycle type corresponding to a non-homogeneous partition $n=\sum_{i=1}^r i \ai$, with $i,\ai, r \in \N$.  Define $n_i = i \ai$, for $i = 1, \dots, r$, and denote $\mathfrak{pvb}_{n_i,i\ai}^!$ the character (given in 1)) corresponding to the partition $n_i=i + \dots +i$ ($\ai$ summands).  Then:

\begin{equation*}
\pvqz = \prod_{i= 1}^r \mathfrak{pvb}_{n_i,i\ai}^!
\end{equation*}
\end{thm:PvBTheorem}

The algebras $\Hspv$ and $\Hspf$ are quadratic algebras, in that they are graded algebras generated in degree 1 with (homogeneous) relations generated in degree 2.    Their `quadratic dual' algebras $\pv$ and $\pf$ also carry an action of the symmetric group $\Sn$.  The characters of these actions on $\pv$ and $\pf$ are given in terms of the graded characters for $\pvq$ and $\pfq$, respectively, by means of a `Koszul formula' for graded characters, which extends the standard Koszul formula which relates the Hilbert series encoding the graded dimensions of a Koszul algebra with the Hilbert series of the algebra's quadratic dual:

\newtheorem*{thm:ExtendedKoszulFormulaThm}{Theorem \ref{thm:ExtendedKoszulFormulaThm}}
\begin{thm:ExtendedKoszulFormulaThm}
Let $G$ be a finite group, let $V$ be a finite-dimensional representation of $G$, and let $G$ act diagonally on the (rational) tensor algebra $TV$.  Let $R\subseteq V \otimes V$ be a submodule and suppose $A:=TV / \langle R\rangle$ is a Koszul algebra.

Then $A$ is a graded representation of $G$ whose character satisfies the `Koszul' formula

\begin{equation}
\As(z) \Aqs(-z) = 1
\end{equation}
where $\As(z)$ is the (graded) character of the representation $A$ evaluated at the element $\sigma\in G$ (and similarly for $\Aqs(z)$).
\end{thm:ExtendedKoszulFormulaThm}

In Corollary \ref{pvTheorem} the Koszul formula for graded characters applies to give:

\begin{equation*}
\pvsz = \frac{1}{\mathfrak{pvb}_{n,\sigma}^{!}(-z)}
\end{equation*}

The graded characters for $\pfq$ and $\pf$ are given in Theorem \ref{PfBTheorem} and Corollary \ref{pfTheorem}.

The paper is organized as follows.  In Part 2 we review some background concerning the groups $\PvB$, $\vB$, $\PfB$ and $\fB$, including their relations to the pure and non-pure braid groups $\PB$ and $\B$.  We further review the definition of the associated graded algebras $\pv$ and $\pf$, and their quadratic duals (which correspond to the cohomology algebras of the respective groups), and explain the action of $S_n$ on these algebras.   We also give a very short primer on the notion of representation stability introduced in \cite{C-F}.  Finally we review some basic notions about wreath products of groups and their representations and introduce related notation which will be needed later.

In Part 3 we state and prove our results concerning the $S_n$-action on $\Hipv$ and $\pv$, as well as our results concerning the cohomology group $\Hsv$.  In Part 4 we cover the corresponding ground for $\PfB$ and $\fB$.  In Part 5 we state and prove the extended Koszul formula for characters of a finite group acting on a Koszul algebra.

\section{Background}

\subsection{The Groups \pmb{$\PvB$} and \pmb{$\PfB$}}
\label{GroupDefsSubs}

As the names suggest, the pure virtual braid groups and the pure flat braid groups are closely related to the pure braid groups.  We explain these relationships here, following \cite{BND}, \cite{Bard}, \cite{BEER} and \cite{Lee} (Section 2.4).

Recall that the braid group $\B$ is generated by the symbols $\{\si: i=1, \dots, (n-1)\}$, subject to the Reidemeister III relation and obvious commutativities:

\begin{align*}
\si\sii\si & = \sii\si\sii \\
\si\sj & = \sj\si \quad \quad \text{ for } |i-j|>1
\end{align*}

The generator $\si$ may be interpreted as corresponding to a braid with $n$ strands with the strand in position $i$ crossing over the adjacent strand to the `right' (i.e. the strand in position $(i+1)$), in a `positive' fashion:

\[
\xy
\vtwist~{(-3,3)}{(3,3)}{(-3,-3)}{(3,-3)};
(-12,-3)*{}; (-12,3)*{} **\dir{-};
(12,-3)*{}; (12,3)*{} **\dir{-};
(-7,0)*{\dots}; (7,0)*{\dots}; (-3,-5)*{i}; (3,-5)*{i \negthinspace + \negthinspace 1}
\endxy
\]
(we draw all strands with upwards orientation).

We obtain the (non-pure) \emph{virtual} braid group $\vB$ by adding to $\B$ the generators $\{s_i: i=1, \dots, (n-1)\}$, referred to as virtual crossings.  The $\{s_i\}$ are subject to the symmetric group relations:

\begin{align*}
s_i s_{i+1} s_i & = s_{i+1} s_i s_{i+1} \\
s_i s_j & = s_j s_i \quad \quad \text{ for } |i-j|>1 \\
s_i^2 &= 1
\end{align*}

The $\si$ and the $s_i$ are subject to certain `mixed' relations:

\begin{align*}
s_i\sii^{\pm 1} s_i & = s_{i+1} \si^{\pm 1} s_{i+1} \\
s_i\sj & = \sj s_i \quad \quad \text{ for } |i-j|>1
\end{align*}

In pictures, virtual crossings $s_i$ are drawn as

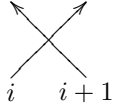
\[
\xy
(0,0)*{} = "1";
(10,0)*{} = "2";
(0,-10)*{} = "3";
(10,-10)*{} = "4";
(0,-12)*{i};
(10,-12)*{i+1};
{\ar@{->} "3"; "2"};
{\ar@{->} "4"; "1"};
\endxy
\]

Hence the pictures corresponding to the mixed relations are as follows (for the case where all ordinary crossings are positive -- one can draw similar pictures when the ordinary crossings are negative):

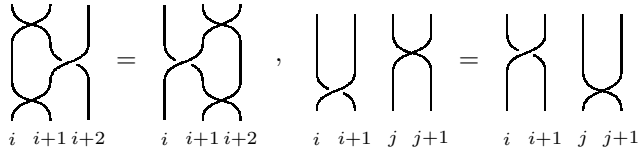
\[
\xy
0;/r.12pc/:
(-45,0)*{\xy
(-20,0)*{\xy
(-5,10)*{\xy (-5,-5)*{}; (5,5)*{} **\crv{(-5,1)&(5,-1)}
\POS?(.5)*{}="x"; (-5,5)*{}; "x" **\crv{(-5,1)}; "x"; (5,-5)*{}
**\crv{(5,-1)} \endxy};
(5,0)*{\xy (-5,-5)*{}; (5,5)*{} **\crv{(-5,1)&(5,-1)}
\POS?(.5)*{\hole}="x"; (-5,5)*{}; "x" **\crv{(-5,1)}; "x"; (5,-5)*{}
**\crv{(5,-1)} \endxy};
(-5,-10)*{\xy (-5,-5)*{}; (5,5)*{} **\crv{(-5,1)&(5,-1)}
\POS?(.5)*{}="x"; (-5,5)*{}; "x" **\crv{(-5,1)}; "x"; (5,-5)*{}
**\crv{(5,-1)} \endxy};
(10,-15)*{}; (10,-5)*{} **\dir{-}; (-10,-5)*{}; (-10,5)*{}
**\dir{-}; (10,5)*{}; (10,15)*{} **\dir{-}
\endxy};
(0,0)*{=};
(20,0)*{\xy
(5,10)*{\xy (-5,-5)*{}; (5,5)*{} **\crv{(-5,1)&(5,-1)}
\POS?(.5)*{}="x"; (-5,5)*{}; "x" **\crv{(-5,1)}; "x"; (5,-5)*{}
**\crv{(5,-1)} \endxy};
(-5,0)*{\xy (-5,-5)*{}; (5,5)*{} **\crv{(-5,1)&(5,-1)}
\POS?(.5)*{\hole}="x"; (-5,5)*{}; "x" **\crv{(-5,1)}; "x"; (5,-5)*{}
**\crv{(5,-1)} \endxy};
(5,-10)*{\xy (-5,-5)*{}; (5,5)*{} **\crv{(-5,1)&(5,-1)}
\POS?(.5)*{}="x"; (-5,5)*{}; "x" **\crv{(-5,1)}; "x"; (5,-5)*{}
**\crv{(5,-1)} \endxy};
(-10,-15)*{}; (-10,-5)*{} **\dir{-}; (10,-5)*{}; (10,5)*{}
**\dir{-}; (-10,5)*{}; (-10,15)*{} **\dir{-};
\endxy}
\endxy};
(-5,0)*{,};
(-75,-20)*{\scriptstyle{i}};(-65,-20)*{\scriptstyle{i+1}};(-55,-20)*{\scriptstyle{i+2}};
(-35,-20)*{\scriptstyle{i}};(-25,-20)*{\scriptstyle{i+1}};(-15,-20)*{\scriptstyle{i+2}};
(5,-20)*{\scriptstyle{i}};(15,-20)*{\scriptstyle{i+1}};(25,-20)*{\scriptstyle{j}};(35,-20)*{\scriptstyle{j+1}};
(55,-20)*{\scriptstyle{i}};(65,-20)*{\scriptstyle{i+1}};(75,-20)*{\scriptstyle{j}};(85,-20)*{\scriptstyle{j+1}};
(45,0)*{\xy
(-25,0)*{\xy
(15,0)*{\xy (-5,-5)*{}; (5,5)*{} **\crv{(-5,1)&(5,-1)}
\POS?(.5)*{}="x"; (-5,5)*{}; "x" **\crv{(-5,1)}; "x"; (5,-5)*{}
**\crv{(5,-1)} \endxy};
(-5,-10)*{\xy (-5,-5)*{}; (5,5)*{} **\crv{(-5,1)&(5,-1)}
\POS?(.5)*{\hole}="x"; (-5,5)*{}; "x" **\crv{(-5,1)}; "x"; (5,-5)*{}
**\crv{(5,-1)} \endxy};
(10,-15)*{}; (10,-5)*{} **\dir{-}; (-10,-5)*{}; (-10,10)*{}
**\dir{-}; (0,-5)*{}; (0,10)*{}
**\dir{-}; (10,5)*{}; (10,10)*{} **\dir{-}; (20,5)*{}; (20,10)*{}
**\dir{-}; (20,-5)*{}; (20,-15)*{}
**\dir{-}
\endxy};
(0,0)*{=};
(25,0)*{\xy
(-5,0)*{\xy (-5,-5)*{}; (5,5)*{} **\crv{(-5,1)&(5,-1)}
\POS?(.5)*{\hole}="x"; (-5,5)*{}; "x" **\crv{(-5,1)}; "x"; (5,-5)*{}
**\crv{(5,-1)} \endxy};
(15,-10)*{\xy (-5,-5)*{}; (5,5)*{} **\crv{(-5,1)&(5,-1)}
\POS?(.5)*{}="x"; (-5,5)*{}; "x" **\crv{(-5,1)}; "x"; (5,-5)*{}
**\crv{(5,-1)} \endxy};
(10,-15)*{}; (10,-15)*{} **\dir{-}; (-10,-15)*{}; (-10,-5)*{}
**\dir{-}; (-10,5)*{}; (-10,10)*{}
**\dir{-};(0,5)*{}; (0,10)*{}
**\dir{-}; (0,-15)*{}; (0,-5)*{}
**\dir{-};(10,-5)*{}; (10,10)*{} **\dir{-}; (20,-5)*{}; (20,10)*{}
**\dir{-};
\endxy}
\endxy}
\endxy
\]

The (non-pure) flat braid group is the quotient of $\vB$ by the relations $\si^2=1$, for all $i=1,\dots,n-1$.

One can show that the $\{\si\}$ generate a copy of the braid group $\B$, while the $\{s_i\}$ generate a copy of the symmetric group $S_n$, within $\vB$.  The map which sends each $\si$ and $s_i$ to $s_i$ gives a surjection $\vB \thra S_n$,whose kernel is by definition the pure virtual braid group, $\PvB$.  A presentation for $\PvB$ may be obtained using the Reidemeister-Schreier method, and the reader is referred to \cite{Bard} for a through explanation, or to \cite{Lee} (Subsection 2.4) for a quick overview.

One finds that $\PvB$ is generated by the set $\{R_{ij}\}_{1 \leq i \ne j \leq n}$, subject to the relations:

\begin{align*}
R_{ij} R_{ik} R_{jk} & = R_{jk} R_{ik} R_{ij} \label{relations1} \\
R_{ij} R_{kl} & = R_{kl} R_{ij}, \quad \quad
\end{align*}
with $i,j,k,l$ distinct.  A typical element $R_{ij}$ may be depicted:

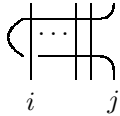
\[
\xy
(0,0)*{}; (-10,0)*{} **\dir{-};
(0,0)*{}; (2,2)*{} **\crv{(2,0)};
(-10,-5)*{}; (-10,0)*{} **\crv{(-14,-2.5)};
(-9,-8)*{}; (-9,2) **\dir{-}  \POS?(.3)*{\hole} = "x";
(-10,-5)*{}; "x" **\dir{-};
"x"; (0,-5)*{} **\dir{-};
(0,-5)*{}; (2,-7)*{} **\crv{(2,-5)};
(2,-7)*{}; (2,-8)*{} **\dir{-};
(-1,-8)*{}; (-1,2)*{} **\dir{-};
(-3,-8)*{}; (-3,2)*{} **\dir{-};
(-6,-2)*{\dots};
(-9,-11)*{i}; (2,-11)*{j}
\endxy
\]
(with all strands oriented upwards).  One can show that the relations satisfied by the $\{s_i\}$ (mixed and unmixed) ensure that one can think of each generator $R_{ij}$ as an ordinary (positive) crossing of strand $i$ over strand $j$, with an arbitrary choice of virtual moves before and after the ordinary crossings to get the strands `into position' (see \cite{BND}).  Hence in pictures one frequently omits the virtual crossings.

The relations in the pure virtual braid group may be illustrated as follows:

\[
\xy
0;/r.12pc/:
(-45,0)*{\xy
(-20,0)*{\xy
(-5,10)*{\xy (-5,-5)*{}; (5,5)*{} **\crv{(-5,1)&(5,-1)}
\POS?(.5)*{\hole}="x"; (-5,5)*{}; "x" **\crv{(-5,1)}; "x"; (5,-5)*{}
**\crv{(5,-1)} \endxy};
(5,0)*{\xy (-5,-5)*{}; (5,5)*{} **\crv{(-5,1)&(5,-1)}
\POS?(.5)*{\hole}="x"; (-5,5)*{}; "x" **\crv{(-5,1)}; "x"; (5,-5)*{}
**\crv{(5,-1)} \endxy};
(-5,-10)*{\xy (-5,-5)*{}; (5,5)*{} **\crv{(-5,1)&(5,-1)}
\POS?(.5)*{\hole}="x"; (-5,5)*{}; "x" **\crv{(-5,1)}; "x"; (5,-5)*{}
**\crv{(5,-1)} \endxy};
(10,-15)*{}; (10,-5)*{} **\dir{-}; (-10,-5)*{}; (-10,5)*{}
**\dir{-}; (10,5)*{}; (10,15)*{} **\dir{-}
\endxy};
(0,0)*{=};
(20,0)*{\xy
(5,10)*{\xy (-5,-5)*{}; (5,5)*{} **\crv{(-5,1)&(5,-1)}
\POS?(.5)*{\hole}="x"; (-5,5)*{}; "x" **\crv{(-5,1)}; "x"; (5,-5)*{}
**\crv{(5,-1)} \endxy};
(-5,0)*{\xy (-5,-5)*{}; (5,5)*{} **\crv{(-5,1)&(5,-1)}
\POS?(.5)*{\hole}="x"; (-5,5)*{}; "x" **\crv{(-5,1)}; "x"; (5,-5)*{}
**\crv{(5,-1)} \endxy};
(5,-10)*{\xy (-5,-5)*{}; (5,5)*{} **\crv{(-5,1)&(5,-1)}
\POS?(.5)*{\hole}="x"; (-5,5)*{}; "x" **\crv{(-5,1)}; "x"; (5,-5)*{}
**\crv{(5,-1)} \endxy};
(-10,-15)*{}; (-10,-5)*{} **\dir{-}; (10,-5)*{}; (10,5)*{}
**\dir{-}; (-10,5)*{}; (-10,15)*{} **\dir{-};
\endxy}
\endxy};
(-5,0)*{,};
(-75,-20)*{i};(-65,-20)*{j};(-55,-20)*{k};
(-35,-20)*{i};(-25,-20)*{j};(-15,-20)*{k};
(5,-20)*{i};(15,-20)*{j};(25,-20)*{k};(35,-20)*{l};
(55,-20)*{i};(65,-20)*{j};(75,-20)*{k};(85,-20)*{l};
(45,0)*{\xy
(-25,0)*{\xy
(15,0)*{\xy (-5,-5)*{}; (5,5)*{} **\crv{(-5,1)&(5,-1)}
\POS?(.5)*{\hole}="x"; (-5,5)*{}; "x" **\crv{(-5,1)}; "x"; (5,-5)*{}
**\crv{(5,-1)} \endxy};
(-5,-10)*{\xy (-5,-5)*{}; (5,5)*{} **\crv{(-5,1)&(5,-1)}
\POS?(.5)*{\hole}="x"; (-5,5)*{}; "x" **\crv{(-5,1)}; "x"; (5,-5)*{}
**\crv{(5,-1)} \endxy};
(10,-15)*{}; (10,-5)*{} **\dir{-}; (-10,-5)*{}; (-10,10)*{}
**\dir{-}; (0,-5)*{}; (0,10)*{}
**\dir{-}; (10,5)*{}; (10,10)*{} **\dir{-}; (20,5)*{}; (20,10)*{}
**\dir{-}; (20,-5)*{}; (20,-15)*{}
**\dir{-}
\endxy};
(0,0)*{=};
(25,0)*{\xy
(-5,0)*{\xy (-5,-5)*{}; (5,5)*{} **\crv{(-5,1)&(5,-1)}
\POS?(.5)*{\hole}="x"; (-5,5)*{}; "x" **\crv{(-5,1)}; "x"; (5,-5)*{}
**\crv{(5,-1)} \endxy};
(15,-10)*{\xy (-5,-5)*{}; (5,5)*{} **\crv{(-5,1)&(5,-1)}
\POS?(.5)*{\hole}="x"; (-5,5)*{}; "x" **\crv{(-5,1)}; "x"; (5,-5)*{}
**\crv{(5,-1)} \endxy};
(10,-15)*{}; (10,-15)*{} **\dir{-}; (-10,-15)*{}; (-10,-5)*{}
**\dir{-}; (-10,5)*{}; (-10,10)*{}
**\dir{-};(0,5)*{}; (0,10)*{}
**\dir{-}; (0,-15)*{}; (0,-5)*{}
**\dir{-};(10,-5)*{}; (10,10)*{} **\dir{-}; (20,-5)*{}; (20,10)*{}
**\dir{-};
\endxy}
\endxy}
\endxy
\]

$\PfB$ has the same presentation as $\PvB$, but subject to the additional relations $R_{ij}R_{ji}=1$.

Recall that, in a similar way the pure braid group $\PB$ is, by definition, the kernel of the group homomorphism $\B \thra S_n$ which sends each $\si$ to $s_i$.  $\PB$ is generated by symbols $\{\Aij: 1\leq i<j \leq n\}$ (for the relations, see e.g. \cite{MarMc}).  As pointed out in \cite{BEER} (Section 4.3), there is a homomorphism $\Psi_n: \PB \ra \PvB$ defined by:

\begin{equation*}
\Aij \mapsto R_{j-1,j} \dots R_{i+1,j} R_{ij} R_{ji} (R_{j-1,j} \dots R_{i+1,j})^{-1}
\end{equation*}

There are also natural homomorphisms $\PB \rightarrow \PvB \rightarrow \PfB$, with the composition being the identity (this observation is due to \cite{BEER}, Subsection 2.3).  The second map, $\Pi_n$, sends all generators $R_{ij}$ to themselves, and the first sends $R_{ij}$ to itself whenever $i<j$.  We conclude that that $\PfB$ is a split quotient of $\PvB$.

One obtains a complex:

\begin{equation*}
1 \ra \PB \overset{\Psi_n}{\lra} \PvB \overset{\Pi_n}{\lra} \PfB \lra 1
\end{equation*}
which is clearly exact on the right, and is shown to be exact on the left in \cite{BEER};  moreover, the kernel of $\Pi_n$ is readily seen to be the normal closure of the image of $\Psi_n$.

\subsection{Associated Graded Algebras of the Group Algebras}

For any finitely presented group $G$, the group ring $\QG$ has a natural augmentation $\QG \thra \Q$ defined by sending each generator to $1\in \Q$.  If one filters the group ring $\QG$ by powers of the augmentation ideal (that is, the kernel of the above augmentation map), the associated graded ring $\grQG$ becomes a natural object of study.  We recall that in \cite{Quillen} it is shown that $\grQG$ is (isomorphic to) the universal enveloping algebra of the Lie algebra $grG \otimes \Q$.

In the case of $\PB$, the associated graded $\pb:=gr\Q\PB$ (often referred to as the chord diagram algebra) is known\footnote{See for instance \cite{Hutchings} and \cite{Kohno}.} to be generated by symbols $\{\aij: 1 \leq i \ne j \leq n; \aij=\aji\}$, subject to the relations:

\begin{align*}
[\aij,\aik+\ajk] &= 0 \\
[\aij,\akl]=0
\end{align*}
for distinct $i,j,k,l$.

In the case of $\PvB$, the associated graded $\pv:=gr\Q\PvB$ is known\footnote{See \cite{BEER} and \cite{Lee}.  In the terminology of \cite{BEER}, $\pv$ is $U(\mathfrak{qtr_n})$, the universal enveloping algebra of the `quasi-triangular' Lie algebra $\mathfrak{qtr_n}$.  The latter is the Lie algebra with the same generators and defining relations as $\pv$. } to be generated by symbols $\{\rij: 1 \leq i \ne j \leq n\}$, subject to the relations:

\begin{align*}
[\rij,\rik+\rjk] + [\rik,\rjk] &= 0 \\
[\rij,\rkl]=0
\end{align*}
for distinct $i,j,k,l$.

Finally, the associated graded $\pf:=gr\Q\PfB$ is known\footnote{See \cite{BEER} and \cite{Lee}.  In the terminology of \cite{BEER}, $\pf$ is $U(\mathfrak{tr_n})$, the universal enveloping algebra of the `triangular' Lie algebra $\mathfrak{tr_n}$, which in turn is the Lie algebra with the same generators and defining relations as $\pf$.} to have the same presentation as $\pv$, but with the extra relations $\rij + \rji =0$.

As in the case of the corresponding groups, one has homomorphisms $\psi_n: \pb \hra \pv$ defined by $\aij \mapsto \rij + \rji$, and $\pi_n: \pv \thra \pf$ which is just the quotient map.  Then again one has a complex

\begin{equation*}
0 \ra \pb \overset{\psi_n}{\lra} \pv \overset{\pi_n}{\lra} \pf \lra 0
\end{equation*}
which is exact on the right and left;  moreover, the kernel of $\pi_n$ is the normal closure of the image of $\psi_n$.\footnote{These observations are due to \cite{BEER}.}

\subsection{The Cohomology Algebras of $\pmb{\PB}$, $\pmb{\PvB}$, $\pmb{\PfB}$ and $\pmb{\PwB}$}
\label{CohomAlgSubsection}

We first recall the definition of quadratic algebras and their duals.

Let $V$ be a finite-dimensional vector space, and denote $TV$ the tensor algebra over $V$ (we assume rational coefficients).  Let $R\subseteq V\otimes V$ be a subspace and denote $\langle R\rangle$ the ideal in $TV$ generated by $R$.  These data permit one to define an algebra $A$ as:

\begin{equation*}
A:= \frac{TV}{\langle R\rangle}
\end{equation*}

\newpage

Such an algebra $A$ is known as a quadratic algebra.\footnote{For more on quadratic algebras, see \cite{PP}.}  The algebra $A$ has a `quadratic dual' $\Aq$ algebra defined as follows.  Let $\Rq\subseteq V^*\otimes V^*$ be the annihilator of $R$.  Then define:

\begin{equation*}
\Aq:= \frac{TV^*}{\langle \Rq\rangle}
\end{equation*}

The algebras $\pb$, $\pv$ and $\pf$ are obviously quadratic.  Their quadratic dual algebras $\pbq$, $\pvq$ and $\pfq$ are known to be the cohomology algebras of the corresponding groups:\footnote{In the case of $\PB$, see \cite{Arnold} and \cite{Kohno}, and in the case of $\PvB$ and $\PfB$ see \cite{BEER} and \cite{Lee}.}

\begin{align*}
\pbq & \cong H^*(\PB,\Q) \\
\pvq & \cong H^*(\PvB,\Q) \\
\pfq & \cong H^*(\PfB,\Q)
\end{align*}

One can readily confirm that $\pbq$ is the exterior algebra generated by $\{\aij: 1 \leq i \ne j \leq n;\ \aij=\aji\}$\footnote{Strictly speaking, the quadratic dual $\pbq$ is generated by the dual generators $\{\aij^*\}$.  However to simplify the notation we will drop the stars.  A similar convention will be adopted for $\pvq$ and $\pfq$.}, subject to the relations:

\begin{equation*}
\aij\ajk + \ajk\aik + \aik\aij= 0
\end{equation*}

Similarly, one finds that the algebra $\pvq$ is the exterior algebra generated by the set $\{\rij: 1 \leq i \ne j \leq n\}$, subject to the relations:

\begin{align}
\rij \wedge \rik & = \rij \wedge \rjk - \rik \wedge \rkj  \notag \\
\rik \wedge \rjk & = \rij \wedge \rjk - \rji \wedge \rik \label{pvbRelations} \\
\rij \wedge \rji & = 0 \notag
\end{align}
where the indices $i,j,k$ are all distinct.

A presentation for the cohomology algebra $\Hspw$ of the pure string motion group was obtained in \cite{J-Mc-M}.  It coincides with that of $\pvq$ given above except without the second line of relations.  Thus $\pvq$ is the quotient of $\Hspw$ by the ideal generated by the second line of relations above.

It is again straightforward to confirm that the algebra $\pfq$ is the exterior algebra generated by the set $\{ \rij : 1\leq i \ne j \leq n,\ \rij = - \rji\}$ subject to the relations:
\begin{align}
\label{pfbRelations}
\rij & \w \rik = \rij\w\rjk \\
\rik & \w \rjk = \rij\w\rjk \notag
\end{align}
for all $i,j,k$ such that $i<j<k$.

Finally, there are homomorphisms $\pi_n^!: \pfq \ra \pvq$ defined by $\rij \mapsto \rij - \rji$, and $\psi_n^!: \pvq \ra \pbq$ defined by $\rij \mapsto \aij$.  It is clear that these maps commute with the action of $S_n$.  One gets a complex:

\begin{equation}
\label{CohomExactSequenceBis}
0 \ra \Hipf \ra \Hipv \ra \Hipb \ra 0
\end{equation}
which is clearly exact on the right.  The fact that the complex is exact on the left follows from Lemma 4.4 of \cite{BEER}, which exhibits an alternative presentation of $\pvq$.  Specifically, one defines symmetric and anti-symmetric parts, $\uij$ and $\vij$, respectively, of the $\rij\in \pvq$; namely $\uij:=\rij+\rji$ and $\vij:=\rij-\rji$.  Then $\pvq$ is the exterior algebra generated by the $\{\vij, \uij\}$ subject to equations like\ (\ref{pfbRelations}) for the $\{\vij\}$, plus:

\begin{align}
\vjk\vij &= \uij\ujk+\ujk\uik+\uik\uij \label{MixedUandVRels} \\
\vij\ujk &= \vik\ujk \notag \\
\vij\uij &= 0 \notag
\end{align}
for distinct $i,j,k$.

The algebra $\pvq$ has a filtration given by deg$(\vij) = 0$ and deg$(\uij)=1$.  Passing to the associated graded $gr(\pvq)$, we find the same relations except that (\ref{MixedUandVRels}) is replaced by

\begin{equation*}
\uij\ujk+\ujk\uik+\uik\uij=0
\end{equation*}

It follows that the map $\phi_n^!: gr(\pvq) \ra \pfq$ given by $\vij \mapsto \rij$ and $\uij \mapsto 0$ is a homomorphic section of $gr(\pi_n^!)$, so that $gr(\pi_n^!)$ (and hence also $\pi_n^!$) is injective.

\subsection{The Action of $\pmb{S_n}$}

The symmetric group $S_n$ acts on the generators $\{ \rij\}$ of $\pv$, $\pvq$, $\pf$ and $\pfq$ via the map $\rij \mapsto r_{\sigma i,\sigma j}$ for $\sigma \in S_n$.  It is easy to check that the relations in the respective algebras are respected by this action, so the action descends to the algebras themselves.

The following definition will be central in our computation of the characters of $\pvq$:

\begin{definition}
\label{CharMonomDef}
Let $m$ be any element in a basis $\mB$ for $\pvq$ (the basis is assumed homogeneous with respect to the grading of $\pvq$).  For any element $\sigma \in S_n$, let $\chism$ be the coefficient of $m$ itself in the expansion of $\sigma(m)$ in terms of the basis $\mB$.  We say that $m$ is characteristic for $\sigma$ if the coefficient $\chism$ is non-zero.
\end{definition}

\newpage

It is clear that:

\begin{equation*}
\pvqsz=\sum_{\chism\ne 0} \chism z^{deg(m)}
\end{equation*}
where the sum is over $m\in \mB$ such that $\chism\ne 0$, and $deg(m)$ is the degree of $m$.  The determination of the character of $\pvq$ will be a matter of determining what are the characteristic monomials and counting their numbers and signs.\footnote{This overall manner of approach is similar to that pursued in \cite{L}, and is very effective in any case where a combinatorially workable basis for the representation is at hand.}

We will adopt a similar definition and notation for the character of $\pfq$.

\subsection{Representation Stability}
\label{RepStabSubsection}

We recall that a family $\{V_n\}$ of groups, or topological spaces, with maps $\phi_n: V_n \ra V_{n+1}$, is said to be cohomologically stable over a ring $R$ if for each $i\geq 1$ the induced map $(\phi_n)^i: H^i(V_{n+1},R) \ra H^i(V_n,R)$ is an isomorphism for $n>> i$.

As mentioned in the Introduction, it was clear already from \cite{Arnold} that for each $i\geq 1$, the dimensions of the algebras $\Hipb$ grow without bound as $n$ increases.  The same is true for the algebras $\Hipv$ and $\Hipf$, as follows from the Hilbert series for the dimensions of these algebras found in \cite{BEER}, reproduced in Equations (\ref{pvqHilbertSeries}) and (\ref{pfqHilbertSeriesEqn}) below.  Hence these algebras are not cohomologically stable.

In \cite{C-F} an alternative to cohomological stability was proposed, referred to as representation stability.  A key feature of their approach is a new notation for the irreducible representations of $S_n$, relative to which the decomposition of representation stable families of $S_n$-representations $\{V_n\}$ stabilizes for sufficiently large $n$.  Recall that for each integer $n\geq 2$, the irreducible representations of $S_n$ are indexed by partitions of $n$.  Given a sequence of integers $n_1\geq \dots \geq n_r \geq 0$ and an integer $n\geq 2n_1+ n_2+ \dots + n_r$, \cite{C-F} denote by $V(n_1, \dots, n_r)_n$ (or just $V(n_1, \dots, n_r)$ if the $n$ is understood) the irreducible $S_n$-representation corresponding to the partition given by the sequence $n-\sum_i n_i \geq n_1\geq \dots \geq n_r$.

As applied to the context of $\PvB$ and $\PfB$, the value of this notation is illustrated by the following computation:

\begin{align}
H^1(\PvBii,\Q) &= V(0) \oplus V(1) \notag \\
H^1(\PvBiii,\Q) &= V(0) \oplus V(1)^2 \oplus V(1,1) \notag \\
H^1(\PvB,\Q) &= V(0) \oplus V(1)^2 \oplus V(1,1) \oplus V(2) \quad \text{for } n\geq 4 \label{pvq1DecompEqn}
\end{align}

In a similar vein, one can show that:

\begin{align}
H^1(\PfBii,\Q) &= V(1) \label{pfqiDecompA} \\
H^1(\PfB,\Q) &= V(1) \oplus V(1,1) \quad \text{for } n\geq 3 \label{pfqiDecompB}
\end{align}

It turns out that similar stability patterns hold for all $i\geq 1$, for both $\PvB$ and $\PfB$, as will be seen in Corollary \ref{pvqRepStabCor} and Corollary \ref{pfqRepStabCor}.

The following definition of uniform representation stability (for the case of $S_n$-representations) is from \cite{C-F}.

\begin{definition}
Let $\{V_n\}$ be a sequence of $S_n$-representations, equipped with linear maps $\phi_n: V_n \ra V_{n+1}$, making the following diagram commutative for each $g\in S_n$:

\[
\xy
(0,0)*+{V_n} = "1";
(15,0)*+{V_{n+1}} = "2";
(0,-15)*+{V_n} = "3";
(15,-15)*+{V_{n+1}} = "4";
{\ar@{->}^{\phi_n\ } "1"; "2"};
{\ar@{->}^{g} "2"; "4"};
{\ar@{->}_g "1"; "3"};
{\ar@{->}_{\phi_n} "3"; "4"};
\endxy
\]
where $g$ acts on $V_{n+1}$ by its image under the standard inclusion $S_n \hra S_{n+1}$.  Such a sequence is called consistent.

We say that the sequence $\{V_n\}$ is uniformly representation stable if:

\begin{enumerate}

\item Injectivity:  The maps $\phi_n: V_n \ra V_{n+1}$ are injective, for sufficiently large $n$.

\item Surjectivity:  The span of the $S_{n+1}$-orbit of $\phi_n(V_n)$ is all of $V_{n+1}$, for sufficiently large $n$.

\item Multiplicities:  Decompose $V_n$ into irreducible $S_n$-representations as

\begin{equation*}
V_n= \bigoplus_{\lambda} \cln V(\lambda)
\end{equation*}
with multiplicities $0\leq \cln \leq \infty$.  Then there is an $N$, not depending on $\lambda$, so that for $n\geq N$ the multiplicities $\cln$ are independent of $n$ for all $\lambda$.  In particular, for any partition $\lambda$ corresponding to the sequence $\lambda_1 \geq  \dots \geq \lambda_r \geq 0$ for which $V(\lambda)_N$ is not defined (that is, $N < \lambda_1 + \sum_i \lambda_i$), $\cln=0$ for all $n\geq N$.
\end{enumerate}
\end{definition}

A key tool in proving uniform representation stability is a result of Hemmer \cite{H} (special cases of which had previously been conjectured by T. Church and B. Farb).  We give the presentation of this result which appears in \cite{C-F}.  Let $H$ be a subgroup of the symmetric group $S_k$, and let $V$ be any representation of $H$.  For $n\geq k$ we extend the action of $H$ on $V$ to an action of $H \times S_{n-k} < S_n$ by letting $S_{n-k}$ act trivially on $V$; this representation is denoted $V \otimes \Q$.

\begin{theorem} (Hemmer, \cite{H}).  Fix $r\geq 1$, a subgroup $H< S_r$, and a representation $V$ of $H$.  Then the sequence of $S_n$-representations $\{\ind_{H\times S_{n-r}}^{S_n} (V \otimes \Q)\}$ is uniformly representation stable.  The decomposition of this sequence stabilizes once $n\geq 2r$.
\label{HemmerThm}
\end{theorem}

In fact Hemmer only explicitly addressed the stabilization of the multiplicities, but as pointed out in \cite{C-F}, the injectivity and surjectivity requirements for representation stability follow from the fact that $\ind_{H\times S_{n-k-j}}^{S_n} V \otimes \Q$ is the $S_n$-span of $V\otimes \Q$ within $\ind_{H\times S_{n-k-j+1}}^{S_{n+1}} V \otimes \Q$.

\subsection{Reminders Concerning \pmb{$S_p \wr S_q$} and Plethysm}
\label{WreathProdSubsection}

We give some reminders concerning the wreath products $S_p\wr S_q$ and their representations, which we will need later.  References for this material include \cite{Sag}, \cite{M} and \cite{J-K}.

The action of $S_q$ on $[q]:=\{1, \dots, q\}$ naturally induces an action of $S_q$ on the product

\begin{equation}
\label{SpqNotationEqn}
S_{p,1} \times \dots \times S_{p,q}
\end{equation}
(where the $S_{p,i}$ are copies of $S_p$ labeled by $i\in [q]$).  Specifically, an element $\tau\in S_q$ acts on the above product by

\begin{equation*}
\tau \cdot (g_1, \dots, g_q) := (g_{\tau^{-1}(1)},\dots, g_{\tau^{-1}(q)})
\end{equation*}
for $g_i\in S_{p,i}$.  The wreath product $S_p\wr S_q$ is then defined as:

\begin{equation*}
S_p\wr S_q := \big( \prod_i S_{p,i} \big) \rtimes S_q
\end{equation*}

We write a general element of $S_p\wr S_q$ as:

\begin{equation}
\label{GeneralEltWreathEqn}
(g_1, \dots, g_q,\tau)
\end{equation}
with $g_i \in S_p$ and $\tau\in S_q$, and we denote by $1\wr S_q$ the subgroup of $S_p\wr S_q$ of elements with $g_i=1$ for all $i=1,\dots, q$.  The projection $\tilde{\omega}: S_p\wr S_q \thra 1\wr S_q$, defined by:

\begin{equation*}
(g_1,\dots,g_q,\tau) \mapsto (1,\dots,1,\tau)
\end{equation*}
is readily seen to be a homomorphism.  Its kernel is the product $\big(\prod_i S_{p,i} \big) \times \{1\}$; in particular this latter subgroup is normal.

We can naturally view $S_p \wr S_q$ as a subgroup of $S_{pq}$.  Indeed, for $i\in [q]$, define $[p]_i:=[ip]\setminus [(i-1)p]$. Then for $l,j\in\Z$ such that $l,l+j\in [q]$, there are bijections $\mu_{l,l+j}: [p]_l \ra [p]_{l+j}$ given by $x \mapsto x+jp$.  By identifying $[p]=[p]_1$ with $[p]_i$ via $\mu_{1,i}$, we may view $S_{p,i}$ (by definition, the group of permutations of $[p]$) as the group of permutations $S_{[p]_i}$ of $[p]_i$.  This gives us an embedding of the subgroup $\prod_i S_{p,i} \times \{1\}$ of $S_p\wr S_q$ into $S_{pq}$.  This embedding is compatible with the action of $S_q$, so we get an embedding of $S_p\wr S_q$ itself into $S_{pq}$.  One can show that $S_p \wr S_q$ is the normalizer of $\prod_i S_{p,i}$ in $S_{pq}$.

\newpage

We now introduce the notion of the wreath product of two modules, $W$ and $V$, over $S_p$ and $S_q$ respectively:

\begin{equation*}
W\wr V := \widetilde{W^{\otimes q}} \otimes \hat{V}
\end{equation*}
where the right-most tensor product is an inner product of modules over $S_p\wr S_q$, and:

\begin{itemize}

\item $\hat{V}$ is the $S_p\wr S_q$-module structure induced on $V$ by the composite $\omega: S_p\wr S_q \overset{\tilde{\omega}}{\ra} 1\wr S_q \overset{\sim}{\ra} S_q$ (where the last map is given by $(1,\dots,\tau) \mapsto \tau$, for $\tau \in S_q$), and

\item $\widetilde{W^{\otimes q}}$ is the vector space tensor product $W^{\otimes q}$ with $S_p\wr S_q$ action given by:

\begin{equation*}
(g_1,\dots, g_q,\tau)\cdot (w_1 \otimes \dots \otimes w_q) = g_1 w_{\tau^{-1}(1)} \otimes \dots \otimes g_q w_{\tau^{-1}(q)}
\end{equation*}

\end{itemize}

Note that we have used the notation $\otimes$ for both inner and outer tensor products.

For the remaining material in this subsection, we assume familiarity with the Frobenius characteristic of a module, and plethysm (see \cite{M}, section I.8).  We now turn to a key relation between the characteristic of a module induced from a wreath product of modules, and the plethystic composition of their characteristics.  Recall that if $g=g(x_1,\dots,x_t)$ is a polynomial which can be expressed as $g(x_1,\dots,x_t)=\sum_{i=1,\dots,k} m_i$ where the $m_i$ are monic monomials in the $\{x_j\}$ (with repeats allowed), and $f=f(y_1,\dots,y_k)$ is a symmetric polynomial in $k$ ($=$ the number of monic monomials in $g$) variables, then the plethysm $f[g]$ is defined as:

\begin{equation*}
f[g]= f(m_1,\dots,m_k)
\end{equation*}

Since $f$ is symmetric, the order in which we numbered the monomials in $g$ does not matter.  Then we have the following relations (ch(-) will denote the characteristic of a module)\footnote{For both relations, see \cite{M}.}:

\begin{equation}
\label{plethysm}
\ch \big( \ind_{S_p\wr S_q}^{S_{pq}} W\wr V \big) = \ch V [\ch W]
\end{equation}

Similarly

\begin{equation}
\label{plethysmB}
\ch \big( \ind_{S_p\times S_q}^{S_{p+q}} W \otimes V \big) = \ch W \cdot \ch V
\end{equation}

\section{Algebras Related to the Pure Virtual Braid Groups}

\subsection{Basis}

Monomials in $\pvq$ may conveniently be represented by graphs on the vertex set $[n]:=\{1, \dots , n\}$, with generators $\rij$ being represented by a directed edge or arrow from $i$ to $j$.  A given graph specifies a unique monomial up to sign.

In \cite{BEER} it was shown that $\pvq$ is a Koszul algebra\footnote{See the last section of this paper for a brief reminder on Koszul algebras.} and has the Hilbert series:

\begin{equation}
\label{pvqHilbertSeries}
\pvqz = \sum_{i=0}^n L(n,i) z^{(n-i)}
\end{equation}
where the $L(n,i)$ stand for Lah numbers (counting the number of unordered partitions of $[n]$ into $i$ ordered subsets).  In particular, $dim \ \mathfrak{pvb}_n^{!n-1} = n!$.

In \cite{Lee} a basis is given for $\pvq$ which makes clear the dependance on Lah numbers, and we recall this now.  A monomial in $\pvq$ (and its corresponding graph) is called admissible if the monomial is of the form $r_{i_1i_2}\w r_{i_2i_3}\w \dots \w r_{i_{m-1}i_m}$ with distinct $i_l$, $1\leq i_l \leq n$.  The set $\{i_1,i_2, \dots ,i_m\}$ is called the support of the monomial, and $i_1$ the root\footnote{This terminology is borrowed from \cite{BEER}, where it was applied to a basis for $\pfq$.}.  Thus admissible graphs are oriented chains of the form:

\[
\xy
(-5,0)*{} = "1";
(0,0)*{} = "2";
(15,0)*{} = "3";
(20,0)*{} = "4";
(-5,2)*{\scriptstyle{i_1}};
(0,2)*{\scriptstyle{i_2}};
(20,2)*{\scriptstyle{i_m}};
{\ar@{->} "1"; "2"};
{\ar@{.} "2"; "3"};
{\ar@{->} "3"; "4"};
\endxy
\]

The following is Theorem 7 from \cite{Lee} (in slightly different language):

\begin{proposition}
\label{BasisPropPvB}
Products of admissible monomials with disjoint supports (in the order of increasing roots) form a basis $\mB$ for $\pvq$.
\end{proposition}

The following lemma is straightforward:

\begin{lemma}
\label{CharMonLemma}
Any collection $\Gamma$ of admissible graphs with disjoint supports determines a unique basis element in $\mB$, namely the product of the corresponding monomials, ordered by increasing roots.  If the union of the supports of $\Gamma$ has cardinality $\alpha$ and $\Gamma$ has $\beta$ components, the degree of the basis element of $\PvB$ determined by $\Gamma$ is $(\alpha - \beta)$.
\end{lemma}

In light of the lemma, for notational simplicity we often conflate such a $\Gamma$ and the basis element it determines.

\newpage

\subsection{$\pmb{S_n}$ Representation on Top Degree Component of $\pvq$}

We can easily determine the nature of the top-degree representation $\pv^{!n-1}$:

\begin{theorem}
The top-degree representation $\pv^{!n-1}$ is (isomorphic to) the regular representation of $S_n$, for all $n$.
\end{theorem}

The proof will be an immediate consequence of the following lemma:

\begin{lemma}
The top-degree representation $\pv^{!n-1}$ has character $\pvs^{!n-1}=n!$, for $\sigma = 1$, and $\pvs^{!n-1}=0$, for $\sigma \ne 1$.
\end{lemma}

\begin{proof}
If $\Gamma$ is a basis element and $\Gamma \in \pvqni$, then $\Gamma$ has just 1 connected component (see Lemma \ref{CharMonLemma}), which must be an admissible graph (that is, an oriented chain).  Every permutation $\sigma$ sends an oriented chain to another oriented chain $\sigma\Gamma$, which is therefore also an admissible graph.  But then\footnote{See the definition of $\chism$ in Definition \ref{CharMonomDef}.} we can only have $\chis(\Gamma)=1$ or $\chis(\Gamma)=0$, and the former holds if and only if $\sigma=1$.  Moreover, we know\footnote{See Equation (\ref{pvqHilbertSeries}) and the subsequent comments.} that $dim\ \pvqni=n!$, and the lemma follows.
\end{proof}

The Theorem is then follows immediately from the basic representation theory of the symmetric group.

\subsection{The Action of $S_n$ on $\pvq$}

We will write $(-1)^i$ for the 1-dimensional representation (of any symmetric group $S_p$) which is either trivial or alternating, according to the parity of $i$.

\begin{theorem}
\label{thm:pvqkDecompositionThm}
As an $S_n$-module, $\pvqk$ has the following decomposition:\footnote{See Subsection \ref{WreathProdSubsection} for notation related to wreath products.}

\begin{equation}
\label{thm:pvqkDecompositionEqn}
\pvqk \cong \bigoplus_{\substack{\alpha \vdash n \\ l(\alpha)=n-k}} \ind_{\prod_i 1\wr \Sai}^{S_n} \bigotimes_i \ \ois (-1)^{i-1}
\end{equation}
where the direct sum is over partitions $\alpha= \sum_{i=1}^{n} i \alpha_i$ of $n$, with $\alpha_i$ non-negative, and with `length' $l(\alpha) = n-k$ parts.  Also:

\begin{enumerate}
\item In the product $\prod_i 1\wr \Sai$, the $1\wr \Sai$ are to be understood as the subgroups $1 \wr \Sai \leq S_i \wr S_{\alpha_i} \leq S_{i\ai}$ and $S_{i\ai}$ is viewed as the subgroup of $S_n$ which permutes the indices $\{1+\sum_{j=1}^{i-1} j \alpha_j, \cdots, i\ai +\sum_{j=1}^{i-1} j \alpha_j\}$;
\item For each $i=1,\dots,n$, $\ois (-1)^{i-1}$ is the pullback of the representation $(-1)^{i-1}$ of $\Sai$ along the projection $\oi: 1\wr \Sai \thra \Sai$, given by $(1,\dots,1,\tau) \mapsto \tau$, for $\tau\in \Sai$; and
\item $\bigotimes_i \ \ois (-1)^{i-1}$ is the outer product of its components.
\end{enumerate}
\end{theorem}

The form of this theorem (decomposition of representations into a sum, over partitions, of certain $1$-dimensional induced representations) is similar to decompositions of the cohomology of the pure braid groups and of the pure welded braid groups, appearing for instance in \cite{O-S}, \cite{L-S}, \cite{C-F} and \cite{W},  which is perhaps not surprising considering how closely related these groups are to $\PvB$.  The overall strategy of the proof is also similar, and in particular owes much to the above papers.

\begin{proof}[Proof of the Theorem]

Note first that a basis element in $\mB$ induces a partition of $[n]$, specifically the partition given by the vertex sets of the basis element's connected components.  Let $\mS$ be a partition of $[n]$ into $(n-k)$ subsets,  and let $\pvqs$ be the vector space generated by those basis graphs whose induced partition is $\mS$.  We let $\bmS$ be the partition of $n$ whose parts are the sizes (that is, the number of vertices) of the subsets in $\mS$.

\begin{example}
\label{PartitionOf4Exa}
Let $n=4$, and consider the partition $\mS=\{1,2\} \sqcup \{3,4\}$.  Then $\pvqs$ is generated by:

\begin{equation*}
r_{12}\w r_{34}, \quad r_{12}\w r_{43}, \quad r_{21}\w r_{34}, \quad r_{21}\w r_{43}
\end{equation*}
and $\bmS$ is the partition $4=2+2$.
\end{example}

Since $S_n$ relabels the vertices in a graph without changing the sizes of the connected components, the $S_n$-orbit of $\pvqs$ consists of the spaces $\pvqsp$ where $\mS'$ is a partition of $[n]$ with $\overline{\mS'}=\bmS$.  So $\pvqk$ decomposes as:

\begin{align}
\pvqk &= \bigoplus_{\substack{\alpha \vdash n \\ l(\alpha)=n-k}} \bigoplus_{\mS: \bmS = \alpha} \pvqs \notag \\
&=  \bigoplus_{\substack{\alpha \vdash n \\ l(\alpha)=n-k}} \ind_{\stab(\mSa)}^{S_n} \pvqsa \label{SumOfInducedRepsEqn}
\end{align}
where $\alpha \vdash n$ means $\alpha$ is a partition of $n$, $l(\alpha)=n-k$ means $\alpha$ has $n-k$ parts, and if $\alpha$ is the partition $n=\sum_{i=1}^n i\ai$ then $\mSa$ is the partition of $[n]$ given by $S_1 \sqcup \dots \sqcup \mS_n$, where

\begin{equation}
\label{ShiftedSets}
\mS_i := \mathrm{Shift}\big( \{1, \dots, i\} \sqcup \{i +1, \dots, 2i \} \sqcup \dots \sqcup  \{(\ai-1)i+1, \dots, i\ai\} \big)
\end{equation}
and Shift$( - )$ means we shift all the indices up by $\sum_{l=1}^{i-1} l\alpha_l$.

The stabilizer $\stab(\mSa)$ can permute the $\alpha_i$ subsets of size $i$, or permute the elements of each subset.  Hence $\stab(\mSa)$ is a direct product of wreath products:

\begin{equation*}
\stab(\mSa)=\prod_{i=1}^n S_i \wr \Sai
\end{equation*}

As a representation of $\stab(\mSa)$, the space $\pvqsa$ is clearly isomorphic to an outer tensor product

\begin{equation}
\label{StabProdEqn}
\pvqsa \cong \pvqsat := \mathfrak{pvb}_{1\alpha_1}^{!\mS_1} \otimes \mathfrak{pvb}_{2\alpha_2}^{!\mS_2} \otimes \dots \otimes \mathfrak{pvb}_{n\alpha_n}^{!\mS_n}
\end{equation}
where the $\mS_i$ were defined in (\ref{ShiftedSets}), and we take the spaces  $\mathfrak{pvb}_{i\alpha_i}^{!}$ to be supported on Shift$(\{ 1, \dots, i\ai \})$.

Finally

\begin{equation}
\ind_{\stab(\mSa)}^{S_n} \pvqsa \cong \ind_{\prod_{i=1}^n S_i \wr \Sai}^{S_n} \otimes_i \mathfrak{pvb}_{i\alpha_i}^{!\mS_i} \label{OrbitDecompEqn}
\end{equation}

We now need the following:

\begin{lemma}
$\mathfrak{pvb}_{i\alpha_i}^{!\mS_i} \cong \ind_{1 \wr \Sai}^{S_i\wr \Sai} \ois (-1)^{i-1}$
\end{lemma}

\begin{proof}[Proof of Lemma]

For any $n\geq 2$, we define $\paq$ as the vector space generated by the basis elements in $\pvq$ whose graph is a forest of trees in which the vertices appear in increasing order (the `a' in $\paq$ refers to this `ascending' order, and we refer to graphs corresponding to basis elements in $\paq$ as ascending graphs).  In terms of the representation of admissible monomials as $r_{i_1i_2}\w r_{i_2i_3}\w \dots \w r_{i_{m-1}i_m}$, this means $i_1 < \dots < i_m$.

Given a (homogeneous) partition $\alpha:  n= i+\dots +i$ ($\ai$ copies), the subspace $\mathfrak{pab}_n^{!\mSa}$ (generated by the ascending basis graph $m$ whose induced partition is $\mSa$) is a one-dimensional space on which an element $\sigma \in 1\wr S_{\ai}$ acts by relabeling the vertices of the basis element $m$ according to $\sigma$.  As a result, while the same monomials appear in $\sigma m$ as in $m$ (since $\sigma \in 1\wr S_{\ai}$), the order of the components of length $i$ is permuted according to $\oi(\sigma) \in \Sai$.  The cost of bringing the permuted components of the basis element back to position (with increasing roots) is precisely the sign $(-1)^{(i-1)|\oi \sigma|}$ (where $|\oi\sigma|$ is the parity of $\oi\sigma$).  Thus $\mathfrak{pab}_{i\ai}^{!\mS_i} \cong \ois (-1)^{i-1}$ as modules over $1\wr\Sai$.

The product $\prod_{j=1}^{\ai} S_{i,j}\times \{1\}$ gives a complete set of coset representatives for $S_i \wr \Sai / 1 \wr \Sai$.  Moreover it is clear that (taking $\mathfrak{pab}_{i\ai}^{!\mS_i}$ to be supported on Shift$(\{ 1, \dots, i\ai \})$):
\begin{equation*}
\mathfrak{pvb}_{i\ai}^{!\mS_i} = \bigoplus_{\sigma\in \prod_{j=1}^{\ai} S_{i,j} \times \{1\}} \sigma \cdot \mathfrak{pab}_{i\ai}^{!\mS_i}
\end{equation*}
from which it follows that $\mathfrak{pvb}_{i\ai}^{!\mS_i} = \ind_{1 \wr \Sai}^{S_i \wr \Sai} \ois (-1)^{i-1}$, as required.

\end{proof}

From this Lemma and (\ref{OrbitDecompEqn}) it follows immediately that
\begin{align}
\ind_{\stab(\mSa)}^{S_n} \pvqsa &\cong \ind_{\prod_i S_i \wr S_{\ai}}^{S_n} \ind_{\prod_i 1 \wr S_{\ai}}^{\prod_i S_i \wr S_{\ai}}\otimes_i \ois (-1)^{i-1} \label{IteratedInducedRp} \\
&\cong \ind_{\prod_i 1 \wr S_{\ai}}^{S_n} \otimes_i \ois (-1)^{i-1} \label{InducedRep}
\end{align}
as claimed in the theorem.

\end{proof}

We now give an alternative description of the decomposition of $\Hkpv$ into irreducibles, in terms of plethysm of Frobenius characteristics. To begin with, claim that

\begin{equation*}
\ind_{\prod_i 1 \wr S_{\ai}}^{\prod_i S_i\wr\Sai} \bigotimes_i \ois (-1)^{i-1} \cong \bigotimes_i \mathrm{Reg}_i \wr \ois (-1)^{i-1}
\end{equation*}
where $\mathrm{Reg}_i$ is the regular representation of $S_i$ (and see subsection \ref{WreathProdSubsection} for the definition of the wreath product of modules).  Indeed, if we write $G_i:=S_i\wr\Sai$, $K_i:=\prod_{j=1}^{\ai} S_{i,j} \times \{1\}$ and $H_i:=1 \wr S_{\ai}$, then $K_i$ gives a set of coset representatives for $G_i / H_i$ and:

\begin{align*}
\ind_{H_i}^{G_i} \ois (-1)^{i-1} &\cong G_i \otimes_{H_i} \ois (-1)^{i-1} \\
&\cong K_i \otimes_{H_i} \ois (-1)^{i-1} \\
&\cong (\otimes_{j=1}^{\ai}\ \Q S_{i,j}) \otimes_{H_i} \ois (-1)^{i-1}
\end{align*}
which, as a vector space, has the correct form.  If $u_i$ is basis vector for the module $\ois (-1)^{i-1}$, the action of $(g_1,\dots,g_{\ai},\tau) \in G_i$ on $(h_1,\dots,h_{\ai},1) \otimes_{H_i} u_i \in K_i \otimes_{H_i} \ois (-1)^{i-1}$ is as follows:

\begin{align*}
(g_1,\dots,g_{\ai},\tau) \cdot \big( (h_1,\dots,h_{\ai},1) \otimes_{H_i} u_i \big) &= \big( (g_1,\dots,g_{\ai},\tau) \cdot (h_1,\dots,h_{\ai},1) \big) \otimes_{H_i} u_i \\
&= (g_1 h_{\tau^{-1} 1},\dots, g_{\ai} h_{\tau^{-1} \ai},\tau) \otimes_{H_i} u_i \\
&= \big( (g_1 h_{\tau^{-1} 1},\dots, g_{\ai} h_{\tau^{-1} \ai},1)\cdot (1,\dots,1,\tau) \big) \otimes_{H_i} u_i \\
&= (g_1 h_{\tau^{-1} 1},\dots, g_{\ai} h_{\tau^{-1} \ai},1) \otimes_{H_i} \big( (1,\dots,1,\tau) \cdot u_i \big) \\
&= (g_1 h_{\tau^{-1} 1},\dots, g_{\ai} h_{\tau^{-1} \ai},1) \otimes_{H_i} (-1)^{|\tau|(i-1)} u_i
\end{align*}
from which the claim follows immediately.

Using this claim and Equation (\ref{IteratedInducedRp}), one can then establish the following description of the characteristic $\mathrm{ch} \Hkpv$ of $\Hkpv$:

\begin{theorem}
\label{thm:pvkPlethysm}
\begin{equation*}
\mathrm{ch} \Hkpv = \bigoplus_{\substack{\ua \vdash k \\ l(\ua) \leq n-k}} \prod_t v_{t,a_t} [\mathrm{ch Reg}_{t+1}] \cdot h_{n-k-l(\ua)}
\end{equation*}
where the sum is over partitions $\ua: k=\sum_{t=1}^k ta_t,\ a_t \in \N$, of $k$, with length $l(\ua)$ (= number of parts in $\ua$) no more than $n-k$; and:

\begin{itemize}
\item $h_p=s_{(p)}$ is the Schur function for the partition $p=p$ of $p\in \N$;
\item $v_{t,a_t}:=h_{a_t}$ if $t$ is even, and $v_{t,a_t}:=e_{a_t}$ if $t$ is odd, where $e_p=s_{(1,\dots, 1)}$ is the Schur function for the partition of $p\in \N$ consisting of $p$ ones;
\item $v_{t,a_t}[\mathrm{ch Reg}_{t+1}]$ denotes plethysm of characteristics.
\end{itemize}
\end{theorem}

The proof is similar to the proof of the corresponding plethystic expression for $\Hkpf$ in Theorem \ref{plethysm}, and so will not be repeated here.  We note that, by using this theorem, it is straightforward to derive the computation of $H^1(\PvB,\Q)$ given in (\ref{pvq1DecompEqn}).

We now establish the representation stability of $\Hkpv$.

\begin{corollary}
\label{pvqRepStabCor}
The cohomology modules $\Hkpv \cong \pvqk$ are uniformly representation stable, for $n\geq 4k$, and $k\geq 1$.
\end{corollary}

\begin{proof}  The proof given in \cite{C-F} with regard to the cohomology of the pure braid groups carries over essentially without change, so we only sketch the proof.

Suppose $\alpha$ is a partition of $n$ into $n-k$ parts, and assume $\alpha$ has $j$ parts (strictly) greater than $1$.  Then we automatically have $j\leq k$, otherwise the parts of $\alpha$ add up to at least $2j+ 1\cdot (n-k-j) = n-k+j >n$, a contradiction.  Hence, in particular, $k+j\leq 2k$.

We have $\prod_i 1\wr\Sai \cong \big( \prod_{2\leq i\leq n} 1\wr\Sai \big) \times S_{n-k-j}$.  The factor $S_{n-k-j}$ just permutes the singleton sets in the partition $\mSa$ of $[n]$, so its action is trivial.  Thus the term corresponding to $\alpha$ in the decomposition (\ref{thm:pvqkDecompositionEqn}) is of the form $\ind_{H\times S_{n-k-j}}^{S_n} V \otimes 1$, where $H=\prod_{2\leq i\leq n} 1\wr\Sai \leq S_n$ and $V$ is a representation of $H$.  Moreover, $k+j\geq 1$.  Therefore we can apply Theorem \ref{HemmerThm}.  We conclude that each summand in (\ref{thm:pvqkDecompositionEqn}) is uniformly representation stable, and stabilizes once $n\geq 2(k+j)$.  Since $k+j\leq 2k$, we conclude that each summand in (\ref{thm:pvqkDecompositionEqn}) stabilizes once $n\geq 4k$.
\end{proof}

\subsection{The Cohomology of $\vB$}
\label{CohomvBSubsec}

In this section we use information about the multiplicity of the trivial representation of $S_n$ in $\pvq$ to compute the rational cohomology of $\vB$.  The general strategy is similar to that used in \cite{C-F} concerning the cohomology of the (ordinary) braid groups, and in \cite{W} concerning the cohomology of the braid-permutation groups $\wB$.

Suppose $G$ and $H$ are groups, with $H$ a normal subgroup of $G$ and $[G:H]<\infty$.  Then the exact sequence

\begin{equation*}
1 \ra H \ra G \ra G/H \ra 1
\end{equation*}
induces a `transfer' homomorphism from which one may deduce an isomorphism (see \cite{B}, Prop. III.10.4):

\begin{equation*}
H^k(G,\Q) \cong H^k(H,\Q)^{G/H}
\end{equation*}
where the RHS is the subspace of $H^k(H,\Q)$ invariant under $G/H$.  One concludes, in particular, that the rank of $H^k(\vB,\Q)$ coincides with the multiplicity of the trivial representation of $S_n$ in $\pvq$.

\newpage

From this we will derive the following theorem:

\begin{theorem}
\label{thm:vBCohomolThm}
For $k,n \in \N$ and $n>k\geq 1$,
\begin{equation*}
H^k(\vB,\Q) \cong \Q^{P(k)}
\end{equation*}
where $P(k)$ is the number of partitions of $k$ into at most $n-k$ parts, in which each odd part (if any) appears just once.

If in fact $n\geq 2k$, then $P(k)$ is the total number of partitions of $k$, in which each odd part (if any) appears just once.
\end{theorem}

We recall\footnote{See A006950, \emph{The On-Line Encyclopedia of Integer Sequences}, http://oeis.org/?language=english.} that a generating function for the number of partitions of $k$ with no repeated odd parts is:

\begin{equation*}
P(z)=\prod_{k>0} \frac{(1+z^{2k-1})}{(1-z^{2k})}
\end{equation*}

\begin{proof}
We determine the multiplicity of the trivial representation in $\pvqk$.

Suppose that $\alpha$ is a partition of $n$ with $n-k$ parts.  By Frobenius reciprocity and equation (\ref{InducedRep}), $\ind_{\stab(\mSa)}^{S_n} \pvqsa$ contains (one copy of) the trivial representation precisely when $\otimes_i \ois (-1)^{i-1}$ is the trivial representation of $\prod_i 1 \wr S_{\ai}$.  In turn it is clear that this holds precisely when $|\Sai|=1$ (that is, when $\ai=1$) whenever $i$ is even.

Then by equation (\ref{SumOfInducedRepsEqn}), the multiplicity of the trivial representation in $\pvqk$ is equal to the number of partitions $\alpha$ of $n$ into $n-k$ parts, with no repeated even parts.

Now suppose $\alpha$ has $j$ parts (strictly) greater than $1$, so that $1\leq j \leq n-k$ (we assume $n>1$); and $\alpha$ has $n-k-j$ parts equal to $1$.

We claim that the partitions of $n$ into $n-k$ parts, with $n-k-j$ parts equal to $1$ and no repeated even parts, are in bijective correspondence with the partitions of $k$ into $j$ parts, with $1\leq j \leq n-k$ and no repeated odd parts.  Indeed, given the former, one obtains the latter by subtracting $1$ from each part.  Conversely, given the latter, one obtains the former by adding $1$ to each part, and then adding a further $n-k-j$ ones.

We can now conclude that, in general, $P(k)$ is the number of partitions of $k$ into at most $n-k$ parts, with no repeated odd parts.  Moreover, if in fact $n\geq 2k$ then the restriction on the number of parts is trivially satisfied, so $P(k)$ is the total number of partitions of $k$ with no repeated odd parts.  This proves the theorem.
\end{proof}

\newpage

The following is half of Theorem \ref{thm:AltMultThm} which was stated in the Introduction (the remainder is deferred to Subsection \ref{pfqActionSubs}):

\begin{theorem}\label{HalfAltMultThm}
For all $n>k\geq 1$,
$\Hkpv^{Alt}$ vanishes for sufficiently large $n$ ($n\geq 2(k+1)$ will do.)
\end{theorem}

\begin{proof}
Let $\alpha$ be a partition of $n$ with $n-k$ parts.  By Frobenius reciprocity, we find that $\ind_{\stab(\mSa)}^{S_n} \pvqsa$ contains (one copy of) the alternating representation precisely when $\otimes_i \ois (-1)^{i-1}$ is the alternating representation of $\prod_i 1 \wr S_{\ai}$.

However, as soon as $n\geq 2(k+1)$, it is clear that $\alpha$ must have at least two parts equal to $1$, so that $|S_{\alpha_1}|\geq 2$.  But then the representation $\ois (-1)^{i-1}$ for $i=1$ is trivial, not alternating.  Hence by (\ref{thm:pvqkDecompositionEqn}) the multiplicity of the alternating representation in $\pvqk$ is nil.

\end{proof}

\subsection{Graded Characters of $\pvq$ and $\pv$}

\begin{theorem}
\label{thm:PvBTheorem}

1). Let $\sigma\in S_n$ have cycle type corresponding to a `homogeneous' partition of $n$, that is $n=k+ \dots + k$ (with $\alpha_k$ summands), for some $k, \alpha_k \geq 1$.  Then:

\begin{equation*}
\pvqsz =  \mathfrak{pvb}_{\alpha_k}^!((-1)^{k-1}k z^k)
\end{equation*}
where $\mathfrak{pvb}_{\alpha_k}^!(z)$ is the Hilbert series for $\mathfrak{pvb}_{\alpha_k}^!$ given in (\ref{pvqHilbertSeries}).

2).  Now let $\sigma\in S_n$ have cycle type corresponding to a non-homogeneous partition $n=\sum_{i=1}^r i \ai$, with $i,\ai, r \in \N$.  Define $n_i = i \ai$, for $i = 1, \dots, r$, and denote $\mathfrak{pvb}_{n_i,i\ai}^!$ the character (given in 1)) corresponding to the partition $n_i=i + \dots +i$ ($\ai$ summands).  Then:

\begin{equation*}
\pvqz = \prod_{i= 1}^r \mathfrak{pvb}_{n_i,i\ai}^!
\end{equation*}
\end{theorem}

The following is an immediate consequence of the Theorem, using the `Koszul formula' of Theorem \ref{thm:ExtendedKoszulFormulaThm}:

\begin{corollary}
\label{pvTheorem}
The characters $\pvsz$ are given in terms of the $\pvqsz$ by the Koszul formulas:

\begin{equation*}
\pvsz = \frac{1}{\mathfrak{pvb}_{n,\sigma}^{!}(-z)}
\end{equation*}
\end{corollary}

Since the character of a representation evaluated at a particular $\sigma\in S_n$ depends only on the conjugacy class to which $\sigma$ belongs, we assume that $\sigma$ may be presented as a product of disjoint cycles $\sigma_1 \sigma_2 \dots \sigma_i \dots$ such that each $\sigma_i$ has length $j_i$ where $j_1\leq j_2\leq \dots$, and may be written $(a_i +1, a_i +2, \dots ,a_i +j_i)$, for some integers $a_l$ such that $a_1<a_2<\dots$.

Recall that in Definition \ref{CharMonomDef}, for any monomial $m$ in the basis $\mB$ for $\pvq$, and for any element $\sigma \in S_n$,  we defined $\chism$ to be the coefficient of $m$ itself in the expansion of $\sigma(m)$ in terms of the basis $\mB$; and we said that $m$ is characteristic for $\sigma$ if $\chism\ne0$.

\begin{proof}[Proof of the Theorem]

Let $m$ be a characteristic monomial for $\sigma$, and in particular a basis element for $\pvq$.  Let $\Gamma$ be the corresponding graph.  We claim that $\sigma \Gamma$ is also a basis element (up to sign).  Indeed, the connected components $\gamma$ of $\Gamma$ must be admissible graphs and have distinct supports, and it is clear (since $\sigma$ is a bijection) that the $\sigma \gamma$ must also be admissible chains of $\sigma \Gamma$ with distinct supports, and the claim follows.

Let $\gamma_1$ be any admissible chain of $\Gamma$.  From the previous paragraph we conclude that $\sigma \Gamma= \pm \Gamma$.  Hence either $\sigma$ acts as the identity on $\gamma_1$ or $\sigma \gamma_1 = \gamma_2$ for some admissible chain $\gamma_2\ne \gamma_1$ of $\Gamma$.  In either case, we see that in fact there must exist some integer $l\geq 1$ and distinct admissible chains $\gamma_1, \dots, \gamma_l$ in $\Gamma$ such that $\sigma \gamma_1 = \gamma_2, \dots, \sigma \gamma_l = \gamma_1$.

Suppose $\gamma_1$ (and hence $\gamma_j,\ j=1, \dots, l$) have length $d$.  It is clear that for each $t=1, \dots, d$, the vertices in position $t$ in $\gamma_1, \dots, \gamma_l$ must comprise a single $l$-cycle in $\sigma$.  (Note in particular that no two vertices from the same admissible chain belong to the same cycle of $\sigma$.)  The picture is as follows:

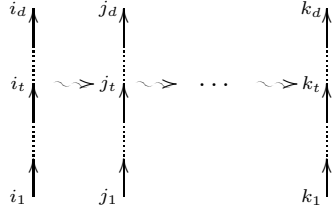
\[
\xy
(-4,-15)*{} = "2";
(-6,-15)*{\scriptstyle{i_1}};
(-4,-10)*{} = "3";
(-4,-5)*{} = "35";
(-4,0)*{} = "4";
(-6,0)*{\scriptstyle{i_t}};
(-4,5)*{} = "5";
(-4,10)*{} = "6";
(-6,10)*{\scriptstyle{i_d}};
{\ar@{->} "2"; "3"};
{\ar@{.} "3"; "35"};
{\ar@{->} "35"; "4"};
{\ar@{.} "4"; "5"};
{\ar@{->} "5"; "6"};
(8,-15)*{} = "12";
(6,-15)*{\scriptstyle{j_1}};
(8,-10)*{} = "13";
(8,-5)*{} = "135";
(8,0)*{} = "14";
(6,0)*{\scriptstyle{j_t}};
(8,5)*{} = "15";
(8,10)*{} = "16";
(6,10)*{\scriptstyle{j_d}};
{\ar@{->} "12"; "13"};
{\ar@{.} "13"; "135"};
{\ar@{->} "135"; "14"};
{\ar@{.} "14"; "15"};
{\ar@{->} "15"; "16"};
(20,0)*{\dots};
(35,-15)*{} = "22";
(33,-15)*{\scriptstyle{k_1}};
(35,-10)*{} = "23";
(35,-5)*{} = "235";
(35,0)*{} = "24";
(33,0)*{\scriptstyle{k_t}};
(35,5)*{} = "25";
(35,10)*{} = "26";
(33,10)*{\scriptstyle{k_d}};
{\ar@{->} "22"; "23"};
{\ar@{.} "23"; "235"};
{\ar@{->} "235"; "24"};
{\ar@{.} "24"; "25"};
{\ar@{->} "25"; "26"};
{\ar@{~>} (-2,0); (4,0)};
{\ar@{~>} (9,0); (15,0)};
{\ar@{~>} (25,0); (31,0)}
\endxy
\]

So we immediately get:

\begin{proposition}
\label{GammaDecomposition}
Let $\sigma$ have cycle type corresponding to some partition $n=\sum_{1\leq i \leq r} i \alpha_i$.  Then $\Gamma$ decomposes into components (not necessarily connected) $\Gamma_1, \dots, \Gamma_r$ such that the $i$-cycles of $\sigma$ permute the chains in $\Gamma_i$ (not necessarily transitively) but are the identity on $\Gamma_j$ for $j\ne i$.
\end{proposition}

The following corollary is clear:

\begin{corollary}
\begin{equation*}
\chi_{\sigma}(\Gamma)=\prod_{i=1}^r \chi_{\sigma}(\Gamma_i)
\end{equation*}
\end{corollary}

Thus we have reduced the problem to the case of $\sigma$ corresponding to  a homogeneous partition (with $\alpha_k$ cycles of size $k$, say).  Consider the vertex $1$, which lies in the cycle $(1 2 \dots k)$.  We label the chain to which $1$ belongs by $\gamma_1$.  From the previous discussion, the vertices of $\gamma_1$ must consist of either $0$ or $1$ vertex from each cycle of $\sigma$.  Once it is known which vertices lie in $\gamma_1$, then the vertices of $\gamma_2:=\sigma \gamma_1, \dots, \gamma_k:=\sigma \gamma_{k-1}$ are determined.

If there are chains in $\Gamma$ beyond $\gamma_1, \dots, \gamma_k$, we consider the remaining chain $\delta_1$ with the smallest vertex; as before,  the vertices of $\delta_1$ must consist of either $0$ or $1$ vertex from each cycle of $\sigma$ which is as yet unaccounted for.  Once it is known which vertices lie in $\delta_1$, there must be remaining chains $\delta_2, \dots, \delta_k$ in $\Gamma$, distinct from the $\gj$, such that $\delta_2:=\sigma \delta_1, \dots, \delta_k:=\sigma \delta_{k-1}$.

We repeat this process until all chains of $\Gamma$ are accounted for.  This leads us to the following proposition.

\begin{proposition}
\label{NuFromMu}
For every basis element $\mu$ in $\mathfrak{pvb}^!_{\alpha_k}$ (where $\alpha_k$ is the number of $k$-cycles in $\sigma$) with some number $\beta$ of components, we get $k^{\alpha_k-\beta}$ distinct possible characteristic monomials $\Gamma$ for $\sigma$, by replacing each vertex $i$ of $\mu$ by a choice of any element in the $i$-th cycle of $\sigma$ (thus producing a basis element $\nu$ in $\pvq$), and then setting $\Gamma=\nu_1 \dots \nu_k$ with $\nu_i:=\sigma^{(i-1)}\nu$.

Moreover, all characteristic monomials of $\Gamma$ arise this way.
\end{proposition}

\begin{proof}
Indeed, there are initially $k^{\alpha_k}$ ways to pick $\nu$, but since cyclic relabelings of the $\nu_1 \dots \nu_k$ correspond to the same graph, we need to divide by a factor of $k$ in respect of each component in $\nu$.  This results in $k^{\alpha_k-\beta}$ distinct choices.  The fact that all characteristic monomials arise this way follows from the discussion immediately preceding Proposition \ref{GammaDecomposition}.
\end{proof}

\begin{lemma}
\label{CharMonDegSign}
Each characteristic monomial $\Gamma$ constructed in Proposition \ref{NuFromMu} satisfies
\begin{equation*}
\sigma\Gamma = (-1)^{(\alpha_k-\beta)(k-1)} \Gamma
\end{equation*}
and has degree $(\alpha_k-\beta)k$.
\end{lemma}

\begin{proof}
Indeed (using the notation from Proposition \ref{NuFromMu}) if $\mu$ has $\beta$ components, then $\mu$ and $\nu$ have degree $(\alpha_k-\beta)$,\footnote{See Lemma \ref{CharMonLemma}.} so $\Gamma$ has degree $(\alpha_k-\beta)k$, as stated.  Also,

\begin{align*}
\sigma\Gamma= \sigma\nu_1 \dots \sigma\nu_k = \nu_2 \dots \nu_k \nu_1 &= (-1)^{\sum_{j=2}^k |\nu_j| |\nu_1|} \nu_1 \dots \nu_k \\
&= (-1)^{(\alpha_k-\beta)^2(k-1)} \nu_1 \dots \nu_k \\
&= (-1)^{(\alpha_k-\beta)(k-1)} \nu_1 \dots \nu_k
\end{align*}
(where $|m|$ is the degree of a monomial $m$; and where we have conflated $\Gamma$ and the $\nu_i$ with the basis elements that they determine, and thus view the above as an equation involving monomials in the anti-symmetric algebra $\pvq$) as required.
\end{proof}

Recall from (\ref{pvqHilbertSeries}) that
\begin{equation*}
\mathfrak{pvb}^!_{\alpha_k}(z)=\sum_{0\leq \beta \leq \alpha_k} L(\alpha_k,\beta) z^{(\alpha_k-\beta)}
\end{equation*}

From Proposition \ref{NuFromMu} and Lemma \ref{CharMonDegSign}, we now see that

\begin{align*}
\mathfrak{pvb}_{n,\sigma}^!(z) &=\sum_{0\leq \beta \leq \alpha_k} L(\alpha_k,\beta) (-1)^{(\alpha_k-\beta)(k-1)} k^{(\alpha_k-\beta)} z^{(\alpha_k-\beta)k} \\
&= \mathfrak{pvb}_{\alpha_k}^!((-1)^{(k-1)} k z^k)
\end{align*}
This completes the proof of the theorem.
\end{proof}

\section{Algebras Related to the Pure Flat Braid Groups}

\subsection{Basis}

Monomials in $\pfq$ may be represented by graphs on the vertex set $[n]:=\{1, \dots , n\}$, with generators $\rij$ being represented by a directed edge or arrow from $i$ to $j$, whenever $i<j$.  As with $\pvq$, a given graph specifies a unique monomial up to sign.

In \cite{BEER} it was shown that $\pfq$ is Koszul, by exhibiting a quadratic Gr\"obner basis for $\pfq$ as an exterior algebra.  They also gave a basis for $\pfq$ itself.

We recall the terminology used in \cite{BEER}.  A monomial in $\pfq$ is called reduced if it has the form $r_{i_1i_2}\w r_{i_1i_3}\w \dots \w r_{i_1i_m}$ with $i_1<i_2< \dots <i_m$.  The set $\{i_1,i_2, \dots ,i_m\}$ is called the support of the monomial, and $i_1$ is its root.  The following is Proposition 4.2 of \cite{BEER}:

\begin{proposition}
Products of reduced monomials with disjoint supports (in the order of increasing roots) form a basis for $\pfq$.
\end{proposition}

We will use a related but slightly different basis for $\pfq$.  Specifically, we will say that a monomial whose graph is a directed chain with indices increasing in the direction of the arrows, is an admissible monomial (and the related graph is an admissible graph).  In terms of the generators, these monomials have the form $r_{i_1i_2}\w r_{i_2i_3}\w \dots \w r_{i_{m-1}i_m}$ with $i_1<i_2< \dots <i_m$.  The set $\{i_1,i_2, \dots ,i_m\}$ is again called the support of the monomial, and $i_1$ the root.  It is easy to see that each subset of $[n]$ determines exactly one reduced monomial and one admissible monomial, and that these are equal modulo the relations, up to sign (we consider that singletons and the empty set determine $1\in\pvq$).  We therefore have:

\begin{proposition}
\label{BasisProp}
Products of admissible monomials with disjoint supports (in the order of increasing roots) form a basis $\mB$ for $\pfq$.
\end{proposition}

We will find it convenient to use the basis induced from admissible monomials, rather than reduced monomials.

\newpage

Lemma \ref{CharMonLemma}, originally stated for $\pvq$, still applies for $\pfq$, and we repeat (and slightly extend) it here for convenience.

\begin{lemma}
\label{BasisPropPfB}
Any collection $\Gamma$ of admissible graphs with disjoint supports determines a unique basis element in $\mB$, namely the product of the corresponding monomials, ordered by increasing roots.  If the union of the supports of $\Gamma$ has cardinality $\alpha$ and $\Gamma$ has $\beta$ components, the degree of the basis element of $\pfq$ determined by $\Gamma$ is $(\alpha - \beta)$.

It follows, in particular, that each partition of $[n]$ determines a unique (and distinct) basis element of $\mB$, since each subset of $[n]$ determines a unique admissible chain.
\end{lemma}

In light of the lemma, for notational simplicity we often conflate such a $\Gamma$ and the basis element it determines.

\subsection{$\pmb{S_n}$ Representation on Top Degree Component of $\pfq$}

We can very easily determine the $S_n$ representation given by the top degree component of $\pfq$.  Indeed, this component is generated by the unique admissible monomial on the full set $[n]$.  It therefore has degree $(n-1)$ and dimension 1.  In fact:

\begin{proposition}
\label{TopDegPfBProp}
The top degree component $\mathfrak{pfb}_n^{!(n-1)}$ of $\pfq$ is the alternating representation of $S_n$.
\end{proposition}

\begin{proof}  Indeed, the element $(12)$ of $S_n$ acts as follows, if $n>2$:

\begin{align*}
(12)\ r_{12}\w r_{23}\w \dots \w r_{(m-1)m} &= r_{21}\w r_{13}\w \dots \w r_{(m-1)m} \\
&= -r_{12}\w r_{13}\w \dots \w r_{(m-1)m} \\
&= -r_{12}\w r_{23}\w \dots \w r_{(m-1)m}
\end{align*}
(where we used the relations (\ref{pfbRelations})), and $(12)\ r_{12}=r_{21}=-r_{12}$ if $n=2$.  Thus $\mathfrak{pfb}_n^{!(n-1)}$ is not the trivial representation, and so (being 1-dimensional) must be the alternating representation.
\end{proof}

The following is an easy corollary:

\begin{corollary}
\label{pfqSignsCor}
Let $S\subseteq [n]$ and let $\gamma$ be the admissible chain in $\pvq$ on the set $S$.  Let $T\subseteq S$ be any subset and $\tau$ be a permutation of $T$. Then:

\begin{equation*}
\tau \gamma = sgn(\tau) \gamma
\end{equation*}
where $sgn(\tau)$ is the sign of $\tau$.
\end{corollary}

\subsection{The Action of $S_n$ on $\pfq$}
\label{pfqActionSubs}

In this subsection we prove the following decomposition of $\Hkpf$ into irreducible $S_n$-modules:

\begin{theorem}
\label{thm:pfqkDecompositionThm}
As an $S_n$-module, $\Hkpf$ has the following decomposition:\footnote{See Subsection \ref{WreathProdSubsection} for notation related to wreath products.}
\begin{equation}
\label{pfqkDecompositionEqn}
\Hkpf \cong \bigoplus_{\substack{\alpha \vdash n \\ l(\alpha)=n-k}} \ind_{\prod_i S_i\wr S_{\alpha_i}}^{S_n} \bigotimes_i \ \mathrm{Alt}_i\wr \ois (-1)^{i-1}
\end{equation}
where the direct sum is over partitions $\alpha= \sum_{i=1}^{n} i \alpha_i$ of $n$, with $\alpha_i$ non-negative, and with length $l(\alpha)=n-k$.  Also:

\begin{enumerate}
\item The $S_i\wr \Sai$ are understood as subgroups of $S_{i\ai}$ in the obvious way, and $S_{i\ai}$ is viewed as the subgroup of $S_n$ which permutes the indices $\{1+\sum_{j=1}^{i-1} j \alpha_j, \cdots, i\ai +\sum_{j=1}^{i-1} j \alpha_j\}$;
    \item $\text{Alt}_i$ is the alternating representation of $S_i$;
\item For each $i=1,\dots,n$, $\ois (-1)^{i-1}$ is the pullback of the representation $(-1)^{i-1}$ of $\Sai$ along the projection $\oi: S_i\wr \Sai \thra \Sai$, which is given by $(\sigma_1, \dots, \sigma_{\ai},\tau) \mapsto \tau$, for $\sigma_j\in S_i$ and $\tau\in\Sai$;
\item $\text{Alt}_i\wr \ois (-1)^{i-1}$ is the wreath product of representations (see Subsection \ref{WreathProdSubsection} for the definition); and
\item $\bigotimes_i$ is the outer product of representations.
\end{enumerate}
\end{theorem}

We note that each component $\mathrm{Alt}_i\wr \ois (-1)^{i-1}$ can alternatively be described as the representation $U_i$, defined as the 1-dimensional representation of $S_i\wr \Sai$ with basis $\{u_i\}$, such that:

\begin{equation*}
(\sigma_1, \dots,\sigma_{\ai},\tau)\cdot (u_i) = (-1)^{\sum_i |\sigma_i| +(i-1)|\tau|} u_i
\end{equation*}
where $|\nu|$ is the parity of any permutation $\nu$.

As in the case of $\pvq$, we have a decomposition

\begin{align}
\pfqk &= \bigoplus_{\substack{\alpha \vdash n \\ l(\alpha)=n-k}} \bigoplus_{\mS: \bmS = \alpha} \pfqs \notag \\
&=  \bigoplus_{\substack{\alpha \vdash n \\ l(\alpha)=n-k}} \ind_{\stab(\mSa)}^{S_n} \pfqsa \notag \\
&=  \bigoplus_{\substack{\alpha \vdash n \\ l(\alpha)=n-k}} \ind_{\prod_i S_i\wr\Sai}^{S_n} \pfqsa \label{pfqOrbitDecompEqn}
\end{align}
with the same notation as in (\ref{SumOfInducedRepsEqn}). The proof is identical to the proof of (\ref{SumOfInducedRepsEqn}) and (\ref{OrbitDecompEqn}), up to obvious name changes, so we do not repeat it.

We now use the fact that, as a representation of the product $\stab(\mSa) \cong \prod_i S_i\wr\Sai$, the space $\pfqsa$ is clearly isomorphic to an outer tensor product

\begin{equation}
\label{pfqTensorProd}
\pfqsat = \mathfrak{pfb}_{1\alpha_1}^{!\mS_1} \otimes \mathfrak{pfb}_{2\alpha_2}^{!\mS_2} \otimes \dots \otimes \mathfrak{pfb}_{n\alpha_n}^{!\mS_n}
\end{equation}
where the $\mS_i$ were defined in (\ref{ShiftedSets}), and we again take the spaces  $\mathfrak{pfb}_{i\alpha_i}^{!}$ to be supported on the sets $\{1, \dots, i\ai\}$ shifted up by $\sum_{l=1}^{i-1} l\alpha_l$.

We need the following lemma:

\begin{lemma}
As modules over $S_i\wr\Sai$,

\begin{equation*}
\mathfrak{pfb}_{i\alpha_i}^{!\mS_i} \cong U_i
\end{equation*}
\end{lemma}

\begin{proof}[Proof of Lemma]
Each space $\mathfrak{pfb}_{i\ai}^{!\mS_i}$ is one-dimensional and is generated by the (unique) basis graph corresponding to the partition $\mS_i$ (see Lemma \ref{BasisPropPfB}).  Suppose the element $(\sigma_1, \dots,\sigma_{\ai},\tau) \in S_i\wr \Sai$ acts on this basis graph.  We have already seen that on each connected component of this graph, $S_i$ acts as the alternating representation (confer Proposition \ref{TopDegPfBProp}).  Hence we get the factor $\sum_{j=1}^{\ai} (-1)^{|\sigma_j|}$.  Moreover, the effect of the $\tau$ is to permute the connected components, so the connected components no longer appear in the order of increasing roots; bringing them back into order (with increasing roots) induces the factor $(-1)^{(i-1)|\tau|}$. The overall factor is then $(-1)^{\sum_i |\sigma_i| + (i-1)|\tau|}$ as required.
\end{proof}

From this Lemma, and Equations (\ref{pfqTensorProd}) and (\ref{pfqOrbitDecompEqn}), it follows immediately that
\begin{equation*}
\pfqk \cong \ind_{\prod_i S_i\wr S_{\alpha_i}}^{S_n} \bigotimes_i U_i = \ind_{S_i \wr S_{\ai}}^{S_{n}} \bigotimes_i \mathrm{Alt}_i\wr \ois (-1)^{i-1}
\end{equation*}
as required.

\begin{corollary}
\label{pfqRepStabCor}
The cohomology modules $\Hipf \cong \pfqi$ are uniformly representation stable, for $n\geq 4i$.
\end{corollary}

As with Corollary \ref{pvqRepStabCor}, the proof given in \cite{C-F} with regard to the cohomology of the pure braid groups carries over essentially without change, and we refer the reader back to Corollary \ref{pvqRepStabCor} for a sketch of the proof.

We end this subsection by showing that, in each degree $i$, the alternating representation has the same multiplicity in $\pfq$ as in $\pvq$(by Theorem \ref{HalfAltMultThm} we know that this multiplicity is nil for large enough $n$).

\newpage

\begin{theorem}
\label{thm:AltMultThm}
For all $n>1$ and $i\geq 1$,
\begin{equation*}
\Hipv^{Alt} \cong \Hipf^{Alt}
\end{equation*}
and both vanish for sufficiently large $n$ ($n\geq 2(i+1)$ will do.)
\end{theorem}

\begin{proof}
By Frobenius reciprocity, the multiplicity of the alternating representation (for $S_n$) in $\ind_{\prod_i S_i\wr S_{\alpha_i}}^{S_n} \bigotimes_i \ \mathrm{Alt}_i\wr \ois (-1)^{i-1}$ is the multiplicity of the alternating representation (for $\prod_i S_i\wr S_{\alpha_i}$) in $\bigotimes_i U_i = \bigotimes_i \ \mathrm{Alt}_i\wr \ois (-1)^{i-1}$.

We can write the semi-direct product $S_i\wr S_{\alpha_i}$ as a product of subgroups, $S_i\wr S_{\alpha_i} = (\prod_{j=1}^{\ai} S_{i,j} \times \{1\}) \cdot (1\wr \Sai)$, with $(\prod_{j=1}^{\ai} S_{i,j} \times \{1\}) \cap (1\wr \Sai) = \{1\}$.  Since the sign function is clearly multiplicative over this product, and the representation $U_i$ restricted to $(\prod_{j=1}^{\ai} S_{i,j} \times \{1\})$ is by definition the alternating representation, it follows that we need only determine the multiplicity of the alternating representation (for $\prod_i 1\wr \Sai$) in $\bigotimes_i \ \ois (-1)^{i-1}$.

But, as per the proof of Theorem \ref{HalfAltMultThm}, this latter is just the multiplicity of the alternating representation (for $S_n$) in $\ind_{\prod_i 1\wr S_{\alpha_i}}^{S_n} \bigotimes_i \ \ois (-1)^{i-1}$, and the claimed isomorphism follows.  The fact that $\Hipv^{Alt}$ vanishes for $n\geq 2(i+1)$ was proved in Theorem \ref{HalfAltMultThm}.
\end{proof}

Using the complex (\ref{CohomExactSequence}) (and the fact that all maps in the complex commute with the action of $S_n$), we obtain the following as a corollary (which was originally derived, by different means, in \cite{H}):

\begin{theorem}
\label{thm:PBAltMultThm}
For all $n>k\geq 1$,
\begin{equation*}
\Hkpb^{Alt} =0
\end{equation*}
\end{theorem}

\subsection{Plethystic View of $\pmb{\pfq}$}

It is clear from Theorem \ref{thm:pfqkDecompositionThm} that understanding induction of representations from wreath products is key to understanding the decomposition of $\pfq$ as an $S_n$-module.  In this subsection we will use the well-known fact that induction from wreath products can be rephrased in terms of plethysm to give an alternate expression for (\ref{pfqkDecompositionEqn}).  We then use this new phrasing to identify a large collection of $S_n$ irreducibles which cannot appear in $\pfq$.  We assume familiarity with the Frobenius characteristic of an $S_n$ module, Schur functions and plethysm; and with the Littlewood-Richardson and Pieri rules for decomposing induced modules.  Two references are \cite{M} (section I.8) and \cite{F-H} (Chapter 1 and Appendix A).

\newpage

\begin{theorem}
\label{thm:Plethysm}
The characteristic $\mathrm{ch} \Hkpf$ of $\Hkpf$ is given by:
\begin{equation*}
\mathrm{ch} \Hkpf = \bigoplus_{\substack{\ua \vdash k \\ l(\ua) \leq n-k}} \prod_t v_{t,a_t} [e_{t+1}] \cdot h_{n-k-l(\ua)}
\end{equation*}
where the sum is over partitions $\ua: k=\sum_{t=1}^k ta_t,\ a_t \in \N$, of $k$, with no more than $n-k$ parts.  Also:

\begin{enumerate}
\item $e_p=s_{(1,\dots, 1)}$ is the Schur function for the partition of $p\in \N$ consisting of $p$ ones;
\item $h_p=s_{(p)}$ is the Schur function for the partition $p=p$ of $p\in \N$;
\item $v_{t,a_t}:=e_{a_t}$ if $t$ is odd, and $v_{t,a_t}:=h_{a_t}$ if $t$ is even;
\item $v_{t,a_t}[e_{t+1}]$ denotes plethysm of Schur functions.
\end{enumerate}
\end{theorem}

\begin{proof}

We refine the decomposition (\ref{pfqkDecompositionEqn}) by singling out, for each partition $\alpha$ of $n$ with $n-k$ parts, the number $j$ of parts greater (strictly) than 1:

\begin{align*}
\Hkpf &= \bigoplus_{\substack{\alpha \vdash n \\ l(\alpha)=n-k}} \ind_{\prod_i S_i\wr S_{\alpha_i}}^{S_n} \bigotimes_i U_i \\
&= \bigoplus_{1\leq j \leq n-k}  \bigoplus_{\substack{\alpha \vdash n \\ l(\alpha)=n-k \\ \alpha_1=n-k-j}} \ind_{(\prod_{i>1} S_i\wr S_{\alpha_i}) \times S_{n-k-j}}^{S_n} \lbrace \big( \bigotimes_{i>1} U_i \big) \otimes 1 \rbrace
\end{align*}
where, as usual, $S_i \wr S_{\alpha_i} \leq S_{i\ai}$ and $S_{i\ai}$ is viewed as the subgroup of $S_n$ which permutes the indices $\{1+\sum_{j=1}^{i-1} j \alpha_j, \cdots, i\ai +\sum_{j=1}^{i-1} j \alpha_j\}$; and $S_{n-k-j}$ permutes the set $\{1,\dots, \alpha_1\}$.

We now unbundle the induction as follows:

\begin{multline}
\label{UnbundleInduction}
\ind_{(\prod_{i>1} S_i\wr S_{\alpha_i}) \times S_{n-k-j}}^{S_n} \lbrace \big( \bigotimes_{i>1} U_i \big) \otimes 1 \rbrace \\
= \ind_{S_{k+j} \times S_{n-k-j}}^{S_n} \lbrace \big( \ind_{(\prod_{i>1} S_i\wr S_{\alpha_i})}^{S_{k+j}} \bigotimes_{i>1} U_i \big) \otimes 1 \rbrace
\end{multline}

The innermost induced module, $\Ka$:

\begin{equation}
\label{KaDef}
\Ka := \ind_{(\prod_{i>1} S_i\wr S_{\alpha_i})}^{S_{k+j}} \bigotimes_{i>1} U_i
\end{equation}
depends only on a partition $\ua$ of $k$ into $j$ parts, which can be obtained from $\alpha$ by subtracting $1$ from each part (so $a_t = \alpha_{t+1}$ for all $t=1,\dots, k$).

Let us consider, for each degree $k$, the modules $\Ka$ that can appear.  We first take the case where $\ua$ is a homogeneous partition of $k$ with more than one part, that is:  $k=ta_t$, for some $1\leq t<k$.

We get:

\begin{equation*}
\Ka = \ind_{S_{t+1}\wr S_{a_t}}^{S_{(t+1)\cdot a_t}}   U_{t+1}
\end{equation*}

If $t$ is odd, we have $U_{t+1} = \text{Alt}_{t+1} \wr \text{Alt}_{a_t}$.  Now if we write $e_p:=\ch \text{Alt}_p$, we obtain by (\ref{plethysm}) that:

\begin{equation}
\label{KaOdd}
\ch \Ka = e_{a_t}[e_{t+1}]
\end{equation}

On the other hand, if $t$ is even, we have $U_{t+1} = \text{Alt}_{t+1} \wr 1_{a_t}$, where for $p$ a positive integer, we denote by $1_p$ the trivial representation for $S_p$.  So if we write $h_p:=\ch 1_p$, we obtain by (\ref{plethysm}) that:

\begin{equation*}
\ch \Ka = h_{a_t}[e_{t+1}]
\end{equation*}

Finally, if we let $v_{t,a_t}:=e_{a_t}$ if $t$ is odd, and $v_{t,a_t}:=h_{a_t}$ if $t$ is even, then we get the following expression for $\ch \Ka$ when $\ua=\sum_t t a_t$ is a non-homogeneous partition of $k$:

\begin{equation}
\label{NonHomogKa}
\ch \Ka = \prod_t v_{t,a_t}[e_{t+1}]
\end{equation}
(this uses (\ref{plethysmB})). A final application of (\ref{plethysmB})) to incorporate the induction:

\begin{equation*}
\ind_{S_{k+j}\times S_{n-k-j}}^{S_n} \Ka \otimes 1
\end{equation*}
concludes the proof.
\end{proof}

We have the following theorem:

\begin{theorem}
\label{thm:LamExceedsK}
The irreducible $V(\lambda)$ occurs in $\pfqk$ only if $|\lambda|\geq k$.\footnote{The notation $V(\lambda)$ for an irreducible representation of $S_n$ was explained in subsection \ref{RepStabSubsection}.  We define $|\lambda|:= \sum_{i\geq 1} \lambda_i$.}
\end{theorem}

\begin{proof}  The theorem will be an immediate consequence of the following lemma:

\begin{lemma}
\label{BallastLemma}
\begin{enumerate}
\item Let $\beta$ be a partition of $p$, and $s_{\beta}$ the corresponding Schur polynomial.  The Schur polynomial corresponding to the irreducible $V(\lambda)$ appears in $s_{\beta}[e_q]$ with positive multiplicity only if $|\lambda | \geq p(q-1)$.
\item For irreducibles $V(\mu)_p$ and $V(\nu)_q$ over $S_p$ and $S_q$ respectively, the irreducible $V(\lambda)$ appears in $\ind_{S_p\times S_q}^{S_{p+q}} V(\mu)_p \otimes V(\nu)_q$ with positive multiplicity only if $|\lambda | \geq |\mu| + |\nu|$.
\end{enumerate}
\end{lemma}

Note that in part 1) we only need the case where $s_{\beta} = h_p$ or $s_{\beta} = e_p$ (in other words, $\beta = (p)$ or $\beta = (1, \dots, 1)$, with $p$ copies of $1$), but nothing turns on this distinction in the proof, so we give the general case.

\begin{proof}[Proof of Lemma]
For part 1), we recall that $s_{\beta}[e_q]$ is the characteristic of a module over $S_{pq}$.  Also, if $\lambda$ is the partition $\lambda: (\lambda_1 \geq \dots \geq \lambda_l>0)$,  then $V(\lambda)=V(\lambda)_{pq}$ is by definition the irreducible $S_{pq}$-representation corresponding to the partition of $pq$ given by $\big((\lambda_0:=pq-\sum_{i=1, \dots, l} \lambda_i) \geq \lambda_1\geq \dots \geq \lambda_l \big)$.  So the claim is equivalent to the claim that $\lambda_0\leq p$.

We view $e_q$ as a polynomial in the variables $\{x_i: i=1,\dots, s\geq 1, \ s\in\N\}$; and suppose we have expanded $s_{\beta}[e_q]$ in terms of the $x_i$.  We claim that, in each monomial of this expansion, $x_1$ never appears with a power greater than $p$.  Otherwise, suppose $m$ is a monomial in the expansion in which $x_1$ is raised to a power greater than $p$.  By the definition of Schur polynomials, $m$ arises from a semi-standard tableau for $\beta$, with the $p$ boxes coloured by the monomials from $e_q$.  It follows that at least one of these boxes must be coloured by a monomial from $e_q$ in which $x_1$ is raised to a power of at least $2$.  But the  monomials in $e_q$, in turn, arise from semi-standard colourings of the Young diagram for the partition $(q)$ of $q$, and hence have at most one power of $x_1$, a contradiction.

Our claim that $\lambda_0\leq p$ is now immediate, since otherwise there would be a semi-standard colouring of the Young diagram for $V(\lambda)$ in which all boxes in the top row are coloured with $x_1$, and the corresponding monomial would have $x_1$ raised to a power greater than $p$.

For part 2), let $\bar{\mu}$, $\bar{\nu}$ and $\bar{\lambda}$ be the partitions of $p$, $q$ and $p+q$ corresponding to $V(\mu)_p$, $V(\nu)_q$ and $V(\lambda)_{p+q}$, respectively; and let $\mu_0$, $\nu_0$ and $\lambda_0$, respectively, be the lengths of the top rows of the corresponding Young diagrams.  Recall that the Littlewood-Richardson rule says that $V(\lambda)$ may appear in $\ind_{S_p\times S_q}^{S_{p+q}} V(\mu)_p \otimes V(\nu)_q$ only if the Young diagram for $\bar{\mu}$ can be expanded to the Young diagram for $\bar{\lambda}$ by a strict $\bar{\nu}$-expansion (see, for instance, \cite{F-H}, Appendix A, for a discussion of the Littlewood-Richarson rule).  It is an immediate consequence of the strictness requirement that $\lambda_0 \leq \mu_0 + \nu_0$, which is obviously equivalent to $|\lambda | \geq |\mu| + |\nu|$.

\end{proof}

If $n=k+j$, it now follows from (\ref{UnbundleInduction}), (\ref{NonHomogKa}) and Lemma \ref{BallastLemma} that $V(\lambda)$ only appears in $\pfqk$ if $|\lambda| \geq \sum_t ta_t = k$, as required.  If $n>k+j$, we still need to determine which $V(\lambda)$ may appear in the induced modules:

\begin{equation}
\label{KaInduced}
\ind_{S_{k+j} \times S_{n-k-j}}^{S_n} \Ka \otimes 1
\end{equation}

By the Pieri rule, $V(\lambda)$ can only appear in this induced module if the Young diagram for the irreducible $V(\lambda)$ can be obtained from the Young diagram for some $V(\mu)$ which appears in $\Ka$ by adding a total of $n-k-j$ boxes to the rows, with no two boxes in the same column (see \cite{F-H}, Appendix A).  But then it is clear that, since $|\mu|\geq k$, we must also have $|\lambda|\geq k$.  This concludes the proof of the theorem.
\end{proof}

\begin{example} Decomposition of $\mathfrak{pfb}_n^{!1}$
\end{example}

The degree $1$ component $\mathfrak{pfb}_n^{!1}$ corresponds to partitions of $n$ given by $(2\geq 1\geq \dots \geq1)$ (with $(n-2)$ copies of $1$).  So the only module $\Ka$ that appears (in the notation of (\ref{KaDef})) is $\Ka=U_2=\text{Alt}_2$.  Then, as per (\ref{KaInduced}), we need to determine the decomposition of:

\begin{equation*}
\text{Alt}_2 \circ 1_{n-2} := \ind_{S_{2} \times S_{n-2}}^{S_n} \text{Alt}_2 \otimes 1
\end{equation*}

We do this by the Pieri rule, which we recalled at the end of the previous theorem.  There are only two ways to extend the diagram for $\text{Alt}_2$ (which we illustrate for $n=6$):

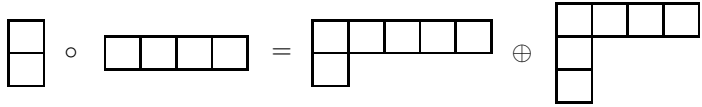
\[
\xy
(10,0)*{\begin{Young}
 \cr
 \cr
\end{Young}} = "1";
(16,0)*{\circ};
(30,0)*{\begin{Young}
 & & &\cr
\end{Young}
} = "2";
(44,0)*{=};
(60,0)*{\begin{Young}
 & & & &\cr
 \cr
\end{Young}} = "3";
(76,0)*{\oplus};
(90,0)*{\begin{Young}
 & & & \cr
 \cr
 \cr
\end{Young}} = "4";
(120,0)*{} = "5";
\endxy
\]

So we conclude that $\mathfrak{pfb}_n^{!1} = V(1) \oplus V(1,1)$ (the second term appearing only for $n\geq 3$), as claimed in (\ref{pfqiDecompA}) and (\ref{pfqiDecompB}).

\begin{theorem}
\label{thm:NoVp}
As a representation of $S_n$, $\pfqk$ does not include the irreducible $V(p)$, for any $p\geq 1$, except for $1$ copy of $V(1)$ in degree $1$.
\end{theorem}

\begin{proof}
We have already covered the degree $1$ case in the previous example, so we may assume $k\geq 2$.

Going back to (\ref{UnbundleInduction}) and (\ref{KaDef}), we again consider the modules $\Ka$ which may appear.  We first assume that $\ua$ is the homogeneous partition of $k$ given by:  $k=1\cdot a_1$ (so $a_1=k$).  By (\ref{KaOdd}),

\begin{equation*}
\ch \Ka = e_{k}[e_{2}]
\end{equation*}

The following result is well known (see \cite{M}, section I.8):

\begin{proposition}
\begin{equation*}
e_k[e_2] = \sum_{\pi} s_{\pi}
\end{equation*}
where the sum is over partitions $\pi$ of $k$ of the form $(\gamma_1-1,\dots,\gamma_r-1\ |\ \gamma_1,\dots, \gamma_r)$, where

\begin{align*}
& \gamma_1 >\dots > \gamma_r >0 \\
\mathrm{and     }\quad  & \sum_i \gamma_i = k
\end{align*}

\end{proposition}

We recall that the notation $(\gamma_1-1,\dots,\gamma_r-1\ |\ \gamma_1,\dots, \gamma_r)$ for the partition $\pi$ of $k$ means that the Young diagram for $\pi$ has $r$ boxes along the main diagonal (with coordinates $(i,i)$, for $1\leq i \leq r$), as well as $\gamma_i -1$ boxes to the right of $(i,i)$, and $\gamma_i$ boxes below $(i,i)$.

Since $k\geq 2$, we have $\gamma_1\geq 2$ also, and so the Young diagram for $\pi$ has at least $3$ rows.  Therefore $V(p)$ (which only has $2$ rows) cannot appear.

We now consider the case where the partition $\ua$ of $k$ has $a_t\geq 1$ for some $t\geq 2$.  Then component $t$ of the product (\ref{NonHomogKa}) is given by $v_{t,a_t}[e_{t+1}]$.  By (the proof of) Lemma \ref{BallastLemma}.1, the Young diagram for each term of the expansion of this plethysm into Schur polynomials has at most $a_t$ columns.  But each such term must correspond to a Young diagram with $a_t(t+1)$ boxes, so the diagram must have at least $t+1\geq 3$ rows.  Finally, it follows easily from the Littlewood-Richardson rule that the Young diagram for each term of the product (\ref{NonHomogKa}) must also have at least $t+1\geq 3$ rows.

This deals with the case where the relevant partition of $n$ in (\ref{UnbundleInduction}) has no parts equal to $1$ (that is, $n=k+j$).  If however $n>k+j$, then as explained at the end of the proof of Theorem \ref{thm:LamExceedsK}, we must add $n-k-j$ boxes to the Young diagrams of irreducibles appearing in the $\Ka$ in accordance with the Pieri rule, in order to determine which irreducibles may appear in $\pfqk$.  But it is clear that adding boxes to a diagram in accordance with the Pieri rule cannot produce a diagram with fewer rows than the original, so we are done.

\end{proof}

\subsection{The Cohomology of $\pmb{\fB}$}
\label{fBCohomSubsection}

The transfer argument described in Subsection \ref{CohomvBSubsec} (and due to \cite{C-F}) carries over to the exact sequence

\begin{equation*}
1 \ra \PfB \ra \fB \ra S_n \ra 1
\end{equation*}

As a result, we know that the rank of $H^i(\fB,\Q)$ coincides with the multiplicity of the trivial representation of $S_n$ in $\pfq$.  From this we will derive the following theorem:

\begin{theorem}
\label{thm:fBCohomThm}
For $i,n \in \N$ and $i\geq 1$,
\begin{equation*}
H^i(\fB,\Q) \cong 0
\end{equation*}
\end{theorem}

\begin{proof}

We adapt a proof given in \cite{W} for the case of the string motion group, and show that all basis elements in $\pfq$ project to zero under the projection $\Pi$ which sends $\pfq$ to its invariant subspace.

For $v \in \pfq$, $\Pi(v)$ is given by the formula:

\begin{equation*}
\Pi(v)= \frac{1}{n!} \sum_{\sigma\in S_n} \sigma(v)
\end{equation*}

If, for each such $v$, there is an involution $\tau\in S_n$ such that $\tau(v)=-v$, then we group the sum according to left cosets of $\langle\tau\rangle$ to conclude that $\Pi(v)=0$:

\begin{equation*}
\Pi(v)= \frac{1}{n!} \sum_{\sigma<\tau>\subseteq S_n} (\sigma(v) + \sigma\tau(v)) = \frac{1}{n!} \sum_{\sigma<\tau>\subseteq S_n} \sigma(v + \tau(v)) =0
\end{equation*}

So suppose $v \in \pfq$ is an element of the basis $\mB$ of $\pfq$ of degree at least $1$.  Then the associated graph has at least one component with at least two vertices.  Suppose the vertices in one such component are $\{i_1 < \dots < i_l\}$ with $l\geq 2$.  Then applying the transposition $(i_1i_2)$:

\begin{align*}
(i_1i_2) \cdot r_{i_1 i_2} \w \dots \w r_{i_{l-1}i_l} &=  r_{i_2 i_1} \w  r_{i_1 i_3} \w \dots \w r_{i_{l-1}i_l} \\
&= -r_{i_1 i_2} \w r_{i_1 i_3} \w \dots \w r_{i_{l-1}i_l} \\
&= -r_{i_1 i_2} \w r_{i_2 i_3} \w \dots \w r_{i_{l-1}i_l}
\end{align*}
as required.

\end{proof}

\subsection{Graded Characters of $\pfq$ and $\pf$}

In \cite{BEER} it was shown that the Hilbert Series for $\pfq$ is

\begin{equation}
\label{pfqHilbertSeriesEqn}
\pfq(z)=\sum_{0\leq k\leq n} S(n,n-k) z^k
\end{equation}
where the $S(n,k)$ are the Stirling numbers of the second kind, which give the number of (unordered) partitions of $[n]$ into $k$ (unordered) subsets.  The character formula that follows generalizes that result.  We let $V(\Gamma)$ denote the set of vertices of a graph $\Gamma$, and let $V(\tau)$ denote the set of indices of any cycle $\tau$ in $S_n$.

We take $\sigma\in S_n$, and as with the case of $\pvq$ we assume that $\sigma$ may be presented as a product of disjoint cycles $\sigma_1 \sigma_2 \dots \sigma_i \dots$ such that each $\sigma_i$ has length $j_i$ where $j_1\leq j_2\leq \dots$, and may be written $(a_i +1, a_i +2, \dots ,a_i +j_i)$, for some integers $a_l$ such that $a_1<a_2<\dots$.

\begin{theorem}
\label{PfBTheorem}
The character $\pfqsz$ is given by:

\begin{equation}
\label{PfBCharEqn}
\pfqsz = \sum_S \prod_{i=1}^r \big[ \sum_{k_i} \epsilon_{k_i} k_i^{(|S_i|-1)} z^{k_i(\sum_{\tau \in S_i} \dt-1)}   \big]
\end{equation}
where
\begin{itemize}

\item the first sum is over the unordered partitions $S=S_1 \sqcup \dots \sqcup S_r$ of the set of cycles (in the cycle decomposition of $\sigma$) into unordered subsets;

\item the second sum is over integers $k_i \geq 1$ which divide the orders of all cycles which belong to $S_i$; that is, such that there exist integers $\dt\geq 1$ satisfying $k_i \dt = |V(\tau)|$ for all cycles $\tau$ in $S_i$; and

\item $\epsilon_{k_i} = (-1)^{(k_i-1)(\sum_{\tau \in S_i} \dt-1) + \sum_{\tau \in S_i} (\dt-1)}$.

\end{itemize}
\end{theorem}

Note that the formula (\ref{PfBCharEqn}) is the case $\sigma=1\in S_n$; indeed, all cycles in $\sigma=1$ have length $1$, so the $k_i$ and $\dt$ are all $1$, and then $\epsilon_{k_i} =+1$.

\begin{corollary}
\label{pfTheorem}
The characters $\pfsz$ are given, in terms of the $\pfqsz$, by the Koszul formulas:

\begin{equation*}
\pfsz = \frac{1}{\mathfrak{pfb}_{n,\sigma}^{!}(-z)}
\end{equation*}
\end{corollary}

As with $\pvq$, for any element $\sigma \in S_n$ and any monomial in the basis $\mB$ for $\pfq$ set out in Proposition \ref{BasisProp}, we let $\chism$ be the coefficient of $m$ itself in the expansion of $\sigma(m)$ in terms of the basis $\mB$; and we say that $m$ is a characteristic monomial for $\sigma$ if the coefficient $\chism$ is non-zero (see Definition \ref{CharMonomDef}). The proof of the theorem is again a matter of determining what are the characteristic monomials and counting their numbers and signs.

\begin{proof}[Proof of Theorem \ref{PfBTheorem}]

Since the defining relations in $\pfq$ are binomial expressions with coefficients $\pm 1$, the expansion of $\sigma(m)$ in terms of $\mB$ will have a single non-zero term and the only questions are whether $\chism\ne 0$ and if so what is the sign.  The following is the key proposition in that regard.

\begin{proposition}
\label{pfbCountingProp}
Let $\sigma\in S_n$ and suppose that every cycle $\tau$ in the cycle decomposition of $\sigma$ consists of consecutive integers, that is $\tau=(\alt + 1, \alt +2, \dots, \alt + |V(\tau)|)$, for some natural number $\alt$.  Suppose $m\in \mB$ satisfies $\chism\ne 0$, and let $\Gamma$ be the corresponding graph.
Then there exist:
\begin{enumerate}
\item a unique unordered partition $S=S_1 \sqcup \dots \sqcup S_r$ of the set of cycles in the cycle decomposition of $\sigma$ (Note: the cycles within any particular $S_i$ need not have the same length); and
\item unique integers $k_i,\ i=1,\dots ,r$, such that $k_i$ divides the orders of all cycles which belong to $S_i$; that is, such that there exist integers $\dt\geq 1$ satisfying $k_i \dt = |V(\tau)|$ for all cycles $\tau$ in $S_i$; and
\item for each $i=1,\dots ,r$ and $\tau \in S_i$, with $\tau =(\alt + 1, \alt +2, \dots, \alt + |V(\tau)|)$, a unique $\ttau \in \{\alt  + 1, \alt +2, \dots, \alt + k_i)$
\end{enumerate}

such that:

\begin{enumerate}[label=\Alph*]

\item $\Gamma$ consists of $r$ components $\Gi, \dots, \Gr$ (not necessarily connected);
\item each $\Gii$ consists of $k_i$ connected components, $\Gii = \gi \ \dots \ \gki$, unique up to cyclic relabeling of the $\gj$;
\item we have $V(\gi) = \bigcup_{\tau\in S_i} \lbrace \ttau +l k_i: l= 0, \dots , (\dt-1) \rbrace$ (and $\gi$ is the unique admissible graph on that collection of indices); and
\item $\gj = \sigma^{(j-1)} \gi, \quad j=1, \dots, k_i$.
\end{enumerate}

\end{proposition}

\newpage

\begin{example}
\end{example}
We illustrate the above proposition by taking the case of

\begin{equation*}
\sigma=(1234)(5)(6789)(10\ 11)(12\ 13\ 14)
\end{equation*}

Suppose we partition the cycles in $\sigma$ into

\begin{equation*}
S=\{(1234),(10\ 11)\} \sqcup \{(5),(6789)\} \sqcup \{(12\ 13\ 14)\}
\end{equation*}

We take $k_1=2, k_2=1, k_3=3$ (noting that these do, indeed, divide the orders of the cycles in parts 1, 2 and 3, respectively, of $S$).  Finally we take $t_{(1234)}=1, t_{(10\ 11)}=11, t_{(5)}=5,t_{(6789)}=6$ and $t_{(12\ 13\ 14)}=13$.

Then the graph which is determined by these data is the following:

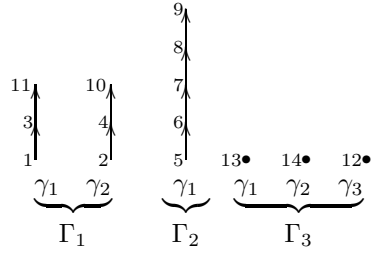
\[
\xy
(-10,0)*{} = "1";
(-10,5)*{} = "2";
(-10,10)*{} = "3";
(-11,0)*{\scriptstyle{1}};
(-11,5)*{\scriptstyle{3}};
(-12,10)*{\scriptstyle{11}};
(0,0)*{} = "4";
(0,5)*{} = "5";
(0,10)*{} = "6";
(-1,0)*{\scriptstyle{2}};
(-1,5)*{\scriptstyle{4}};
(-2,10)*{\scriptstyle{10}};
(10,0)*{} = "7";
(10,5)*{} = "8";
(10,10)*{} = "9";
(10,15)*{} = "12";
(10,20)*{} = "13";
(9,0)*{\scriptstyle{5}};
(9,5)*{\scriptstyle{6}};
(9,10)*{\scriptstyle{7}};
(9,15)*{\scriptstyle{8}};
(9,20)*{\scriptstyle{9}};
(18,0)*{\scriptstyle{\bullet}} = "10";
(26,0)*{\scriptstyle{\bullet}} = "11";
(34,0)*{\scriptstyle{\bullet}} = "111";
(16,0)*{\scriptstyle{13}};
(24,0)*{\scriptstyle{14}};
(32,0)*{\scriptstyle{12}};
(-5,-5)*{\underbrace{\gamma_1 \quad \gamma_2}};
(-5,-10)*{\Gamma_1};
(10,-5)*{\underbrace{\gamma_1}};
(10,-10)*{\Gamma_2};
(25,-5)*{\underbrace{\gamma_1 \quad \gamma_2\quad \gamma_3}};
(25,-10)*{\Gamma_3};
{\ar@{->} "1"; "2"};
{\ar@{->} "2"; "3"};
{\ar@{->} "4"; "5"};
{\ar@{->} "5"; "6"};
{\ar@{->} "7"; "8"};
{\ar@{->} "8"; "9"};
{\ar@{->} "9"; "12"};
{\ar@{->} "12"; "13"};
\endxy
\]

One sees in particular that the components within each $\Gii$ are cyclically permuted by $\sigma$ (with signs).  Specifically,

\[
\xy
(-25,5)*{\sigma(\Gamma_1) \quad =};
(-10,0)*{} = "1";
(-10,5)*{} = "2";
(-10,10)*{} = "3";
(-11,0)*{\scriptstyle{2}};
(-11,5)*{\scriptstyle{4}};
(-12,10)*{\scriptstyle{10}};
(0,0)*{} = "4";
(0,5)*{} = "5";
(0,10)*{} = "6";
(-1,0)*{\scriptstyle{3}};
(-1,5)*{\scriptstyle{1}};
(-2,10)*{\scriptstyle{11}};
{\ar@{->} "1"; "2"};
{\ar@{->} "2"; "3"};
{\ar@{->} "4"; "5"};
{\ar@{->} "5"; "6"};
(10,5)*{=\quad -};
(20,0)*{} = "7";
(20,5)*{} = "8";
(20,10)*{} = "9";
(19,0)*{\scriptstyle{2}};
(19,5)*{\scriptstyle{4}};
(18,10)*{\scriptstyle{10}};
(30,0)*{} = "10";
(30,5)*{} = "11";
(30,10)*{} = "12";
(29,0)*{\scriptstyle{1}};
(29,5)*{\scriptstyle{3}};
(28,10)*{\scriptstyle{11}};
{\ar@{->} "7"; "8"};
{\ar@{->} "8"; "9"};
{\ar@{->} "10"; "11"};
{\ar@{->} "11"; "12"};
(40,5)*{=};
(50,0)*{} = "13";
(50,5)*{} = "14";
(50,10)*{} = "15";
(49,0)*{\scriptstyle{1}};
(49,5)*{\scriptstyle{3}};
(48,10)*{\scriptstyle{11}};
(60,0)*{} = "16";
(60,5)*{} = "17";
(60,10)*{} = "18";
(59,0)*{\scriptstyle{2}};
(59,5)*{\scriptstyle{4}};
(58,10)*{\scriptstyle{10}};
{\ar@{->} "13"; "14"};
{\ar@{->} "14"; "15"};
{\ar@{->} "16"; "17"};
{\ar@{->} "17"; "18"};
\endxy
\]
where the signs are determined by interpreting the diagrams as wedge products (in the indicated order) of the monomials corresponding to their connected components.  Similarly one finds that $\sigma(\Gamma_2)=-\Gamma_2$, and

\[
\xy
(5,0)*{\sigma(\Gamma_3) =};
(18,0)*{\scriptstyle{\bullet}} = "10";
(26,0)*{\scriptstyle{\bullet}} = "11";
(34,0)*{\scriptstyle{\bullet}} = "111";
(16,0)*{\scriptstyle{14}};
(24,0)*{\scriptstyle{12}};
(32,0)*{\scriptstyle{13}};
(40,0)*{\ =\ \Gamma_3};
\endxy
\]

\begin{proof}[Proof of Proposition \ref{pfbCountingProp}]

Since $\sigma(\Gamma)=\pm \Gamma$, it must be possible to group the connected components of $\Gamma$ into collections $\Gi, \dots, \Gr$ such that the connected components within each $\Gii$ are cyclically permuted by $\sigma$.

For any one of these collections $\Gii$, let $\gi, \dots, \gki$ be the connected components of $\Gii$, labeled so that $\sigma(\gi)=\gamma_2, \dots, \sigma(\gki)=\gi$.  Obviously, the numbering of the $k_i$ connected components is determined precisely up to cyclic relabeling of the $\gj$.

Let $\tau$ be a cycle in the cycle decomposition of $\sigma$ such that $V(\tau)\cap V(\Gii)\ne \emptyset$.  Then we must have
\begin{multline*}
\tau\big(V(\tau)\cap V(\gi)\big)=V(\tau)\cap V(\gamma_2), \\
\dots, \\
\tau\big(V(\tau)\cap V(\gki)\big)=V(\tau)\cap V(\gi)
\end{multline*}

Hence $V(\tau) \subseteq V(\Gii)$.  Moreover, the $V(\tau)\cap V(\gj)$ must all have the same size, which we call $\dt$. Thus

\begin{equation*}
|V(\tau)|=k_i \dt
\end{equation*}
and in particular $k_i$ divides $|V(\tau)|$ for every cycle $\tau$ in $\sigma$ such that $V(\tau)\cap V(\Gii)\ne \emptyset$.

Let $\ttau$ be the smallest element of $V(\tau)\cap V(\gi)$.  Then, recalling that we have assumed that $\tau$ (may be presented so that it) is a list of consecutive integers, we must have

\begin{equation*}
\ttau + l \in V(\tau)\cap V(\gamma_{l+1}), \quad \forall\ l=0, \dots, (k_i-1)
\end{equation*}
and then $\ttau + k_i \in V(\gi)$.  Similarly, we find that
\begin{equation*}
\ttau + k_i +l \in V(\tau)\cap V(\gamma_{l+1}), \quad \forall\ l=0, \dots, (k_i-1)
\end{equation*}
and then $\ttau + 2k_i \in V(\gi)$.  Continuing in this way, we conclude that
\begin{equation}
\label{Vtau}
V(\tau)\cap V(\gi) = \{\ttau + sk_i,\ s=0, \dots, \dt-1\}
\end{equation}
and
\begin{equation*}
V(\tau)\cap V(\gj) = \sigma^{(j-1)} \big( V(\tau)\cap V(\gi) \big), \quad j=1, \dots, k_i
\end{equation*}

Thus we see that

\begin{equation*}
V(\gi) = \bigcup_{\tau\in S_i} \lbrace \ttau +s k_i: s= 0, \dots , (\dt-1) \rbrace
\end{equation*}
(and since $\Gamma$ is a basis element, $\gi$ must be the unique admissible graph on that set).  Furthermore,

\begin{equation*}
\gj = \sigma^{(j-1)} \gi \quad j=1, \dots, k_i
\end{equation*}

Moreover, since $\ttau$ is the smallest element of $V(\tau)\cap V(\gi)$, and $V(\tau)=\{\alt+1, \dots, \alt +k_i \dt\}$, for some natural number $\alt$, we conclude from (\ref{Vtau}) that $\ttau \in \{\alt+1, \dots, \alt + k_i \}$.

As noted previously, if $\tau$ is a cycle in $\sigma$ such that $V(\tau)\cap V(\Gii) \ne \emptyset$, then in fact $V(\tau) \subseteq V(\Gii)$. Thus each collection $\Gii$ determines a unique subset of the cycles in $\sigma$, and so the $\{\Gii\}_{i=1,\dots,r}$ determine a unique partition $S=S_1 \sqcup \dots \sqcup S_r$ of the cycles in $\sigma$.

We have thus seen how each characteristic monomial determines uniquely the data described in the proposition, as required.
\end{proof}

\newpage

\begin{proposition}
\label{MonomialsFromDataProp}
For any $\sigma \in S_n$, each possible choice of data as per 1-3 of Proposition \ref{pfbCountingProp}, that is:
\begin{itemize}
\item an unordered partition $S=S_1 \sqcup \dots \sqcup S_r$ of the cycles in the cycle decomposition of $\sigma$;
\item integers $k_i,\ i=1,\dots ,r$, such that $k_i$ divides the orders of all cycles which belong to $S_i$; that is, such that there exist $\dt\geq 1$ satisfying $k_i \dt = |V(\tau)|$ for all cycles $\tau$ in $S_i$; and
\item for each $i=1,\dots ,r$ and $\tau \in S_i$, with $\tau =(\alt + 1, \alt +2, \dots, \alt + |V(\tau)|)$, some $\ttau \in \{\alt  + 1, \alt +2, \dots, \alt + k_i)$
\end{itemize}
gives rise to a unique characteristic monomial of the form described in A-D of that proposition, and these are all distinct (up to cyclic relabelings of the $\gj$ within each $\Gii, i=1, \dots, r$).
\end{proposition}

\begin{proof}

It is fairly clear that given the data 1-3 we can form a graph $\Gamma$ as per A-D of Proposition \ref{pfbCountingProp}, which is unique up to cyclic relabelings of the $\gj, j=1, \dots, k_i$ within each $\Gamma_i$.  Moreover, these are distinct.  Indeed the vertex sets $V(\gamma)$ of the connected components $\gamma$ of $\Gamma$ determine a partition of $[n]$, and different $\Gamma$ determine different partitions.  The claim then follows because Lemma \ref{BasisPropPfB} implies that each partition determines a unique and distinct basis element.

By construction, $\sigma$ just permutes the $\gj, j=1, \dots, k_i$, within each $\Gamma_i$, so that $\sigma(\Gamma)=\pm \Gamma$ and $\Gamma$ is a characteristic monomial for $\sigma$.

\end{proof}

\begin{proposition}
For any $\sigma \in S_n$, and for each possible choice of data as per Proposition \ref{pfbCountingProp}, each component $\Gii$ of the resulting graph satisfies

\begin{equation*}
\sigma(\Gii)=(-1)^{(k_i-1)(\sum_{\tau} \dt-1) + \sum_{\tau} (\dt-1)} \ \Gii
\end{equation*}
and has degree $k_i(\sum_{\tau} \dt-1)$.  Furthermore, there are exactly $k_i^{(|S_i|-1)}$ choices of the $\ttau \in \{\alt  + 1, \alt +2, \dots, \alt + k_i \}$, after factoring out cyclic relabelings of the $\gj$.

Hence the space of characteristic monomials corresponding to any particular partition $S$ of the cycles in the cycle decomposition of $\sigma$ has character

\begin{equation*}
\chi_{\scriptstyle{S}} = \prod_{S_i\in S} \sum_{k_i} (-1)^{(k_i-1)(\sum_{\tau} \dt-1) + \sum_{\tau} (\dt-1)} k_i^{(|S_i|-1)} z^{k_i(\sum_{\tau} \dt-1)}
\end{equation*}
where the sum is over all $k_i$ dividing $|V(\tau)|$ for each cycle $\tau$ in $\sigma$ such that $V(\tau)\cap V(\Gamma_i)\ne \emptyset$.
\end{proposition}

\begin{proof}

We first determine the degree of the monomials corresponding to each $\Gii,\ i=1, \dots, r$.  Recall that for each connected component $\gj$ of $\Gii$, the set $V(\gj)\cap V(\tau)$ has $\dt$ elements.  Hence $V(\gj)$ has $\sum_{\tau}\dt$ elements (the sum being over all $\tau$ such that $V(\gj)\cap V(\tau)\ne \emptyset$) and the degree of the monomial corresponding to $\gj$ is $\sum_{\tau}\dt-1$.  Hence the degree of the monomial corresponding to $\Gii$ is $k_i(\sum_{\tau}\dt-1)$.

Let us now consider any particular $\Gamma_i$.  Note that if $\sigma$ is increasing on (the vertex set of) some admissible chain $\gamma$, then $\sigma\gamma$ remains an admissible graph.  Also, for any cycle $\tau$ in $\sigma$ such that $V(\tau)\subseteq V(\Gii)$, $\tau$ is increasing on each set $V(\gj)\cap V(\tau)$, except for the $\gj$ which contains the biggest element of $\tau$, namely $\alt + k_i \dt$.  In fact $\tau$ is increasing even on this $\gj$, except at $\alt + k_i \dt$, which $\tau$ maps to $\alt + 1$.

So we only need to apply relations to bring this particular $\sigma\gj$ into admissible form.  We recall that the indices in each $\tau$ are consecutive, in the sense that $\tau = (\alt+1, \dots, \alt + |V(\tau)|)$, for some $\alt$;  this implies that $\gj$ consists of subchains, each corresponding to a different $\tau$, with directed edges joining the subchains into a single ascending chain of indices (that is, the chain will have ascending indices as we follow the arrows along the chain).  Since $\gj$ has ascending indices, and $\sigma$ does not `mix' subchains corresponding to different $\tau$, one need only reorder the indices within each subchain of $\sigma\gj$ (so that they are strictly increasing), and then the resulting graph will automatically also have ascending indices. Thus it suffices to determine the required sign to reorder each segment, and then collect the signs.  We illustrate this process in the following picture, for some segment corresponding to a particular $\tau$, showing only the part of $\gj$ that involves indices from $\tau$.  For simplicity of notation we have assumed that $\alt=0$:

\[
\xy
(-5,0)*{\sigma\quad } = "1";
(-4,-10)*{} = "2";
(-4,-5)*{} = "3";
(-4,0)*{} = "4";
(-4,5)*{} = "5";
(-4,10)*{} = "6";
(-4,-12)*{\scriptstyle{k_i}};
(-4,12)*{\scriptstyle{k_i\dt}};
{\ar@{->} "2"; "3"};
{\ar@{->} "3"; "4"};
{\ar@{.} "4"; "5"};
{\ar@{->} "5"; "6"};
(2,0)*{=};
(8,-10)*{} = "12";
(8,-5)*{} = "13";
(8,0)*{} = "14";
(8,5)*{} = "15";
(8,10)*{} = "16";
(8,-12)*{\scriptstyle{k_i+1}};
(8,12)*{\scriptstyle{1}};
{\ar@{->} "12"; "13"};
{\ar@{->} "13"; "14"};
{\ar@{.} "14"; "15"};
{\ar@{->} "15"; "16"};
(13,0)*{=};
(25,0)*{(-1)^{(\dt-1)}};
(35,-10)*{} = "22";
(35,-5)*{} = "23";
(35,0)*{} = "24";
(35,5)*{} = "25";
(35,10)*{} = "26";
(35,-12)*{\scriptstyle{1}};
(35,12)*{\scriptstyle{1+k_i(\dt-1)}};
{\ar@{->} "22"; "23"};
{\ar@{->} "23"; "24"};
{\ar@{.} "24"; "25"};
{\ar@{->} "25"; "26"};
\endxy
\]
(for the sign, see Corollary \ref{pfqSignsCor}.)

Hence in respect of each $\tau$ we get a sign $(-1)^{(\dt -1)}$, and so for $\gj$ (and hence also for $\Gii$) we get a sign $(-1)^{\sum_{\tau} (\dt -1)}$.

Next, recall that, by construction, $\sigma\Gamma_i=\pm \Gamma_i$.  To determine what the sign is, note that

\begin{align*}
\sigma(\gi \dots \gki) &= (-1)^{\sum_{\tau} (\dt-1)} \gamma_2 \dots \gki \gi \\
&= (-1)^{\sum_{\tau} (\dt-1)} (-1)^{(k_i-1)|\gi|^2} \gi \dots \gki \\
&= (-1)^{\sum_{\tau} (\dt-1)} (-1)^{(k_i-1)|\gi|} \gi \dots \gki \\
&= (-1)^{\sum_{\tau} (\dt-1) + (k_i-1)(\sum_{\tau}\dt-1)} \Gii
\end{align*}

Finally, after these steps $\sigma(\Gii)$ has been brought into basis form, that is, a product of admissible graphs with disjoint support and ordered by increasing roots, at the cost of the above signs.

For each part $S_i$ of the partition $S$, and for each $k_i$ dividing the orders of all the cycles $\tau$ such that $V(\tau)\subseteq V(\Gii)$, there are $k_i^{|S_i|}$ ways to pick the $\{\ttau\}$, but because the $\gj$ may be cyclically relabeled without changing the graph, we really only have $k_i^{(|S_i|-1)}$ possible graphs.  Hence the character for the space of characteristic monomials corresponding to the part $S_i$ is

\begin{equation*}
\chi_{\scriptstyle{S_i}} = \sum_{k_i} (-1)^{(k_i-1)(\sum_{\tau} \dt -1) + \sum_{\tau} (\dt -1)} k_i^{(|S_i|-1)} z^{k_i (\sum_{\tau} \dt -1)}
\end{equation*}
where, as usual, the sum is over $k_i$ dividing the orders of all the cycles $\tau$ such that $V(\tau)\subseteq V(\Gii)$; and, for each $k_i$ and each such $\tau$, $|V(\tau)|=k_i\dt$.

We note that the space of characteristic monomials for $\sigma$ induced by the partition $S$ is clearly isomorphic to the tensor product of the spaces of characteristic monomials for each $S_i$, so that

\begin{equation*}
\chi_{\scriptstyle{S}}= \prod_i \chi_{\scriptstyle{S_i}}
\end{equation*}

\end{proof}

Finally, the space of characteristic monomials for $\sigma$ is just the direct sum of the spaces of characteristic monomials induced by the various partitions $S$, so that

\begin{equation*}
\pfqsz= \sum_S \chi_{\scriptstyle{S}}
\end{equation*}

\end{proof}

\section{A Koszul Formula for Graded Characters}

In this section we will state and prove a generalization of the well-known Koszul formula which applies to quadratic algebras which have the `Koszul' property.  We briefly recall the necessary concepts.  Our presentation follows \cite{PP}, to which the reader may refer for further information.

We assume given a quadratic algebra $A$ defined, as in Subsection \ref{CohomAlgSubsection}, by $A:= TV / \langle R\rangle$.  For each $n=2,3,\dots$, and for $1\leq i\leq n-1$, define $\Xni:=V^{\otimes i-1}\otimes R\otimes V^{\otimes n-1-i}$.

One can define a graded complex, known as the Koszul complex, whose degree $n$ component is the following:

\begin{multline}
\label{KoszulComplex}
0 \lra \Xnio \cap \dots \cap \Xnin \overset{d_1}{\lra} \Xnit \cap \dots \cap \Xnin \overset{d_2}{\lra} \frac{X^n_3 \cap \dots \cap \Xnin}{\Xnio} \overset{d_3}{\lra} \dots \notag \\
\dots \overset{d_{i-1}}{\lra} \frac{\Xni \cap \dots \cap \Xnin}{\Xnio + \dots + X^n_{i-2}} \overset{d_i}{\lra} \dots \\
\overset{d_{n-2}}{\lra} \frac{\Xnin}{\Xnio + \dots + X^n_{n-3}} \overset{d_{n-1}}{\lra} \frac{V^{\otimes n}}{\Xnio + \dots + X^n_{n-2}} \overset{d_{n}}{\lra} \frac{V^{\otimes n}}{\Xnio + \dots + \Xnin} \lra 0 \notag
\end{multline}
where we write $U/V$ for $U/(U\cap V)$.

The map $d_i$ is the composition of the obvious inclusion and projection:

\begin{equation*}
\frac{\Xni \cap \dots \cap \Xnin}{\Xnio + \dots + X^n_{i-2}} \hra \frac{\Xnii \cap \dots \cap \Xnin}{\Xnio + \dots + X^n_{i-2}} \thra \frac{\Xnii \cap \dots \cap \Xnin}{\Xnio + \dots + X^n_{i-1}}
\end{equation*}
With these $d_i$ it is easy to check that the previous sequence is a complex, for each $n$.

The algebra $A$ is said to be Koszul when the Koszul complex is exact for all $n\geq 2$.  For $A$ Koszul, one can show\footnote{See, for instance, \cite{PP}, Cor. 2.2.2.} that

\begin{equation}
\label{KoszulFormula}
A(z) \Aq(-z) = 1
\end{equation}
where $A(z)$ is the Hilbert series encoding the dimensions of the graded components of $A$ (and similarly for $\Aq(z)$).

The Koszul formula (\ref{KoszulFormula}) has the following generalization:

\begin{theorem}
\label{thm:ExtendedKoszulFormulaThm}
Let $G$ be a finite group, let $V$ be a finite-dimensional representation of $G$, and let $G$ act diagonally on the (rational) tensor algebra $TV$.  Let $R\subseteq V \otimes V$ be a submodule and suppose $A:=TV / \langle R\rangle$ is a Koszul algebra.

Then $A$ is a graded representation of $G$ whose character satisfies the `Koszul' formula

\begin{equation}
\label{ExtendedKoszulFormula}
\As(z) \Aqs(-z) = 1
\end{equation}
where $\As(z)$ is the (graded) character of the representation $A$ evaluated at the element $\sigma\in G$ (and similarly for $\Aqs(z)$).
\end{theorem}

The usual Koszul formula (\ref{KoszulFormula}) is just the case where $\sigma=1$, the identity of $G$.

\begin{proof}
Since $R$, and hence $\langle R\rangle$, is a submodule of $TV$, the fact that $A$ is a $G$-module is clear.  Moreover, since $V$ and $R$ are $G$-modules, so are the $\Xni$, as well as their various intersections, sums and quotients, such as:

\begin{align*}
\Xni \cap \dots \cap \Xnin \\
\Xnio + \dots + X^n_{i-2} \\
E_i:=\frac{\Xni \cap \dots \cap \Xnin}{\Xnio + \dots + X^n_{i-2}}
\end{align*}

The kernel of $d_{i+1}$ is (by exactness of the Koszul complex) the subspace:

\begin{equation*}
F_i:= \frac{\Xni \cap \dots \cap \Xnin}{\Xnio + \dots + X^n_{i-1}} \subseteq E_{i+1}
\end{equation*}
where we again write $U/V$ for $U/(U\cap V)$.

By the discussion above, $F_i$ is in in fact a submodule of $E_{i+1}$.  Hence, by Maschke's theorem, there is a submodule $F_{i+1} \subseteq E_{i+1}$ such that:

\begin{equation*}
E_{i+1} / F_i \cong F_{i+1} \text{    and    } E_{i+1}=F_i\oplus F_{i+1}
\end{equation*}
as $G$-modules.

Hence the Koszul complex is isomorphic to the sequence of modules:

\begin{equation*}
0 \ra F_1 \overset{d_1}{\ra} F_1\oplus F_2 \overset{d_2}{\ra} F_2\oplus F_3 \overset{d_3}{\ra} \dots \ra F_{n-1} \oplus F_n \overset{d_n}{\ra} F_n \ra 0
\end{equation*}
with $d_{i+1}(F_i)=0$ and $d_{i+1}(F_{i+1})\cong F_{i+1}$.  If we write $\chi_i$ for the character of $F_i$ evaluated at $\sigma$, and $\chi_0=\chi_{n+1}=0$, it is trivial that

\begin{equation*}
\sum_{i=0}^{n} (-1)^i (\chi_{i} + \chi_{i+1}) =0
\end{equation*}

But, by the additivity of characters of direct sums of modules, $(\chi_{i} + \chi_{i+1})$ is the character of $E_{i+1}=F_i \oplus F_{i+1}$.  Moreover, one knows that

\begin{equation*}
E_{i+1}= A^{!n-i}\otimes A^i
\end{equation*}
(this can be seen by inspection, but see also \cite{PP}, Prop. 1.6.2 and Prop. 2.3.1).

Hence, by the multiplicativity of characters of tensor products of modules, $(\chi_{i} + \chi_{i+1}) = A_{\sigma}^{!n-i} A_{\sigma}^i$ ($A_{\sigma}^i$ is the character of the representation $A^i$ evaluated at $\sigma$, and similarly for $A_{\sigma}^{!(n-i)}$).  So we find that:

\begin{equation}
\label{KoszulAuxEquation}
\sum_{i=0}^{n} (-1)^i A_{\sigma}^{!n-i} A_{\sigma}^i =0
\end{equation}
for $n\geq 2$.  In fact, the same equation (\ref{KoszulAuxEquation}) clearly holds also for $n=0,1$ (the case $n=0$ corresponding to the trivial representation $A^0=A^{!0}=\Q$).  Since equation (\ref{KoszulAuxEquation}) is just the degree $n$ part of equation (\ref{ExtendedKoszulFormula}), the result follows.

\end{proof}

\end{document}